\documentclass{amsart}
\pdfoutput=1

\usepackage{amssymb}
\usepackage{amsthm}
\usepackage{amsfonts}
\usepackage{amsmath}
\usepackage{pmboxdraw}
\usepackage{verbatim} 
\usepackage{graphicx}
\usepackage{color}
\usepackage[colorlinks=true, citecolor=blue, filecolor=black, linkcolor=black, urlcolor=black]{hyperref}
\usepackage{cite}
\usepackage[normalem]{ulem}
\usepackage{subcaption}
\usepackage{todonotes}
\usepackage{bbm}
\usepackage{kantlipsum}
\allowdisplaybreaks

\newcommand{\RN}[1]{%
	\textup{\uppercase\expandafter{\romannumeral#1}}%
}

\def\pa{\partial}

\def\R{\mathbb{R}}

\newcommand{\erf}{\operatorname{erf}}

\newcommand{\sgn}{\operatorname{sgn}}
\newcommand{\Tr}{\operatorname{Tr}}

\theoremstyle{plain}
\newtheorem*{thm*}{Theorem}
\newtheorem{thm}{Theorem}[section]
\newtheorem{lem}[thm]{Lemma}
\newtheorem{cor}[thm]{Corollary}
\newtheorem{prop}[thm]{Proposition}
\newtheorem*{prop*}{Proposition}
\newtheorem*{lem*}{Lemma}

\theoremstyle{definition}
\newtheorem*{eg*}{Example}
\newtheorem*{egs*}{Examples}

\newtheorem{ex}[thm]{Example}
\newtheorem{rem}[thm]{Remark}
\newtheorem{rem*}[thm]{Remark}

\theoremstyle{remark}
\newtheorem*{rmk*}{Remark}
\newtheorem*{rmks*}{Remarks}

\numberwithin{equation}{section}

\begin{document}
\title[Harer-Zagier type recursion formula for the elliptic G{\SMALL in}OE]{Harer-Zagier type recursion formula for the elliptic G{\SMALL in}OE}

\author{Sung-Soo Byun}
\address{Department of Mathematical Sciences and Research Institute of Mathematics, Seoul National University, Seoul 151-747, Republic of Korea}
\email{sungsoobyun@snu.ac.kr}

\date{\today}

\thanks{Sung-Soo Byun was partially supported by the POSCO TJ Park Foundation (POSCO Science Fellowship), by the New Faculty Startup Fund at Seoul National University and by the National Research Foundation of Korea funded by the Korea government (NRF-2016K2A9A2A13003815, RS-2023-00301976).
}

\begin{abstract}
We consider the real eigenvalues of the elliptic Ginibre matrix indexed by the non-Hermiticity parameter $\tau \in [0,1]$, and present a Harer-Zagier type recursion formula for the even moments in the form of an $11$-term recurrence relation. 
For the symmetric GOE case ($\tau=1$), it reduces to a known 5-term recurrence relation. On the other hand, for the asymmetric cases when $\tau < 1$, the recursion formula is new, even in the special case of the well-studied Ginibre ensemble ($\tau=0$), where it reduces to a 3-term recurrence. For the proof, we derive a seventh-order linear differential equation for the moment generating function.
\end{abstract}

\maketitle

\section{Introduction and main results}

Random Matrix Theory (RMT) enjoys an intimate connection with various branches of mathematics and physics \cite{ABD11}. One prominent illustration of this relationship is the Harer-Zagier formula \cite{HZ86}, which stands as a well-known example demonstrating the combinatorial and topological significance inherent in RMT statistics.
While the Harer-Zagier formula originates in the study of the moduli space of curves, it also gives rise to a fundamental formula in the study of spectral moments of classical random matrix ensembles (cf. \cite{Ch11}).

To be more concrete, let us consider the Gaussian Unitary Ensemble (GUE) $X^{\rm GUE}$ picked randomly with respect to the measure proportional to $ e^{ -\frac{1}{2} \Tr X^2 } \,dX $ on the space $\mathcal{H}_N$ of $N \times N$ Hermitian matrices. 
Here, $d X$ is the Lebesgue measure on $\mathcal{H}_N \cong \mathbb{R}^{N^2}$.
One of the most basic observables of the GUE is its $p$-th spectral moment 
\begin{equation}
M_p^{ \rm GUE }:= \mathbb{E} \Big[\Tr (X^{\rm GUE})^p \Big] , \qquad p \in \mathbb{N}.
\end{equation}
Note here that, due to symmetry, all odd moments vanish. On the other hand, it possesses non-trivial even moments. 
For instance, the first few even moments are given by 
$$
M_0^{\rm GUE} = N, \qquad M_2^{\rm GUE} = N^2, \qquad M_4^{\rm GUE}=2N^3+N, \qquad  M_6^{\rm GUE}=5N^4+10N^2.
$$ 
In \cite{HZ86}, it was demonstrated using the Wick calculus that the spectral moments satisfy the $3$-term recursion formula
\begin{equation} \label{GUE recursion}
(p+1) M_{2p}^{ \rm GUE }= (4p-2) N \,M_{2p-2}^{\rm GUE}+(p-1)(2p-1)(2p-3) M_{2p-4}^{\rm GUE}. 
\end{equation}
This recursion formula yields that for any non-negative integer $p,$ 
\begin{equation} \label{GUE genus expansion}
 M_{2p}^{ \rm GUE } = \sum_{g=0}^{ \lfloor p/2 \rfloor } c(g;p) \, N^{p+1-2g}.
\end{equation}
Here, the coefficients $c(g; p)$ can be identified with the number of pairings of the edges of a $2p$-gon, dual to a map on a compact Riemann surface of genus $g$. 
For this reason, a formula of the type \eqref{GUE genus expansion} is commonly referred to as the genus expansion in RMT. 
Indeed, this type of formula provides one of the earliest examples of a topological expansion in the theory of matrix integrals \cite{BIPZ78}.
In particular, the leading coefficient $c(0;p)$ is given by the $p$-th Catalan number
\begin{equation} \label{Catalan}
c(0; p) = \frac{1}{p+1} \binom{2p}{p}.
\end{equation}
Since the Catalan number corresponds to the even moments of the Wigner's semicircle law
\begin{equation} \label{semi circle}
d\mu_{ \rm sc }(x):= \frac{ \sqrt{4-x^2} }{2\pi} \, \mathbbm{1}_{(-2,2)}(x)\,dx,
\end{equation}
this leading coefficient \eqref{Catalan} gives rise to the convergence of the GUE density towards \eqref{semi circle}. 
Furthermore, when combined with the loop equation formalism, the expansion \eqref{GUE genus expansion} can be utilized to derive the large-$N$ expansion of the densities, cf. \cite{WF14}.

The Harer-Zagier formula \eqref{GUE recursion} was revisited by Haagerup and Thorbjørnsen in \cite{HT03}. 
They re-derived this recursion formula using a closed-form expression for the moment-generating function, which is given in terms of a confluent hypergeometric function, cf. \eqref{u tau1}.
Furthermore, they obtained a similar formula for the Laguerre Unitary Ensemble (LUE). 
Additional discussions on the LUE spectral moments can be found in \cite{CDO21,Di03,CMSV16}, and similar work on the Jacobi Unitary Ensemble (JUE) is detailed in \cite{GGR21}.
Notably, for all these classical GUE, LUE, and JUE models, the spectral moments can be evaluated in closed form using hypergeometric polynomials \cite{CMOS19}. 
We also refer to \cite{FLSY23,CCO20,MPS20,Le05,BFO24} for discrete extensions of the above-mentioned results and to \cite{Fo21} for an extension of \eqref{GUE recursion} to the $d$-dimensional Fermi gas in an harmonic trap.
In the context of RMT, the Harer-Zagier type formulas can be employed to derive (small) deviation inequalities for extreme eigenvalues \cite{HT03,Le04,Le09}, which play an important role in the study of the associated universality problems, see e.g. \cite{FS10,EY23}. 
Furthermore, as previously mentioned, it can be used to study the finite-size corrections of the densities or the counting statistics \cite{PS16,Bo16,FT19,YZ23, DDMS19}.
For further mathematical implementation beyond RMT, we refer the reader to a recent work \cite{GGHZ22}, which delves into expressing the Harer-Zagier formula within the context of noncommutative geometry.
Furthermore, for applications of these formulas in the context of the time-delay matrix of quantum dots and $\tau$-function theory, see \cite{MS11,MS12,MS13,Cu15,CMSV16,LV11} and references therein.


Beyond the unitary invariant ensembles, it is natural to investigate the Harer-Zagier type formula for random matrices with orthogonal symmetry. 
The single-most fundamental model in this class is perhaps the Gaussian Orthogonal Ensemble (GOE) $X^{ \rm GOE}$ that follows the probability distribution proportional to $ e^{-\frac14 \Tr X^2} \,dX $, where in this case $d X$ is the Lebesgue measure on the space of symmetric matrices $\mathcal{S}_N \cong \mathbb{R}^{N(N+1)/2}$. 
Similar to the above, we write 
\begin{equation}
M_p^{ \rm GOE }:= \mathbb{E} \Big[\Tr (X^{\rm GOE})^p \Big] , \qquad p \in \mathbb{N}
\end{equation}
for the $p$-th spectral moment of the GOE.
For instance, we have the following explicit evaluations
\begin{align*}
M_0^{ \rm GOE }=N, \qquad M_2^{ \rm GOE }=N^2+N, 
\qquad 
 M_4^{ \rm GOE }=2N^3+5N^2+N, \qquad M_6^{ \rm GOE }=5N^4+22N^3+52N^2+41N. 
\end{align*}
It is trivial but noteworthy for the latter discussion that $M_0^{\rm GUE} = M_0^{\rm GOE} = N$ since they coincide with the number of eigenvalues.

As is widely recognized, the integrable structure of orthogonal ensembles is considerably more intricate when compared to their unitary counterparts (cf. \cite{AFNM00}). 
This complexity leads to a delay in the investigation of the Harer-Zagier type formula for the GOE.  
Notably, Ledoux demonstrated in \cite[Theorem 2]{Le09} that the GOE spectral moments satisfy the following $5$-term recurrence relation
\begin{align}  \label{HZ formula for GOE}
\begin{split}
(p+1) M_{2p}^{ \rm GOE } &= (4p-1)(2N-1) M_{2p-2}^{ \rm GOE } 
 +(2p-3)(10p^2-9p-8N^2+8N) M_{2p-4}^{ \rm GOE }
\\
&\quad -5(2p-3)(2p-4)(2p-5)(2N-1) M_{2p-6}^{ \rm GOE }
\\
&\quad -2(2p-3)(2p-4)(2p-5)(2p-6)(2p-7) M_{2p-8}^{ \rm GOE }.
\end{split}
\end{align}
We also refer to \cite{Mi19} for a combinatorial aspect of this formula.
In \cite{Le09}, the formula \eqref{HZ formula for GOE} follows from a linear differential equation for the moment generating function (MGF), see also 
\cite{RF21} and references therein for more recent work. 
The method of Ledoux relies on elementary yet nontrivial analysis employing Gaussian integration by parts and certain properties of the Hermite polynomials. This method can also be applied to re-derive the recursion \eqref{GUE recursion} for the GUE.
From the viewpoint of the integrable structure, the additional technical difficulty for the GOE, compared to the GUE, is that the $1$-point function \eqref{GOE density v2} of the GOE consists of two parts. 
One of these parts is indeed the Christoffel-Darboux kernel of the Hermite polynomials, which coincides with the GUE density.
For the GUE density part, one can make use of the classical Christoffel-Darboux formula.


While there has been extensive study on the spectral moments of Hermitian random matrix ensembles, their non-Hermitian counterparts remain relatively unexplored, and we aim to make contributions in this direction.
In non-Hermitian random matrix theory, the basic model typically considered is the Ginibre matrix, see \cite{BF22,BF23} for recent reviews. 
In particular, what distinguishes the real Ginibre matrix, referred to as GinOE, from its complex or quaternionic counterparts, GinUE and GinSE respectively, is the presence of purely real eigenvalues, see Figure~\ref{Fig_REGinibre}.
We shall focus on the statistics of real eigenvalues of real random matrices. 
Compared to Hermitian random matrices, the study of real eigenvalues of asymmetric random matrices has additional conceptual and technical challenges:
\begin{itemize}
    \item the number of real eigenvalues is random;
    \smallskip
    \item the classical Christoffel-Darboux formula does not apply. 
\end{itemize}
We refer to \cite{FS23,FIK20,LMS22,Si17,Si17a,AB22,KPTTZ16,BKLL23,BMS23,Fo23,Fo15a,FT21} and references therein for recent work on real eigenvalues of various asymmetric random matrices. 

In this work, we consider the elliptic GinOE, the real random matrices $X \equiv X_\tau$ that are distributed according to the probability measure proportional to
\begin{equation}
 \exp\Big(-\frac{ \Tr(X^T X-\tau X^2) }{2(1-\tau^2)} \Big)\, dX, \qquad \tau\in[0,1], 
\end{equation}
where $dX$ is the Lebesgue measure on the space $\mathbb{R}^{N^2}$ of $N \times N$ matrices with real elements, and $X^T$ is the transpose of $X$.  
Here, the parameter $\tau$ is known as the non-Hermiticity parameter. 
Alternatively, the elliptic GinOE can be defined as 
\begin{equation}\label{X GOE antiGOE}
X_\tau := \sqrt{\frac{1+\tau}{2}}\, S_+ + \sqrt{\frac{1-\tau}{2}} \, S_-, \qquad S_\pm := \frac{1}{\sqrt{2}} (G \pm G^T).
\end{equation}
Here $G$ is an $N\times N$ GinOE matrix, the random matrix with all independent Gaussian entries $\textup{N}[0,1/N]$. 
Notice that the symmetrisation $S_+$ of the GinOE coincides with the GOE.
The elliptic random matrix model stands out as a well-known model, seamlessly bridging fundamental concepts in non-Hermitian and Hermitian random matrix theories.
Namely, for $\tau=0$, it recovers the GinOE, whereas in the limit $\tau\uparrow 1$, it recovers the GOE. 
We refer the reader to \cite[Section 2.3]{BF22} and \cite[Sections 2.8 and 5.5]{BF23} for reviews as well as comprehensive references on the elliptic Ginibre ensembles. See also \cite{OYZ23,FG23,Fo23,HHJK23,BL23} for some very recent works on the elliptic Ginibre matrices.

\begin{figure}[t]
    \centering
    \includegraphics[width=0.8\textwidth]{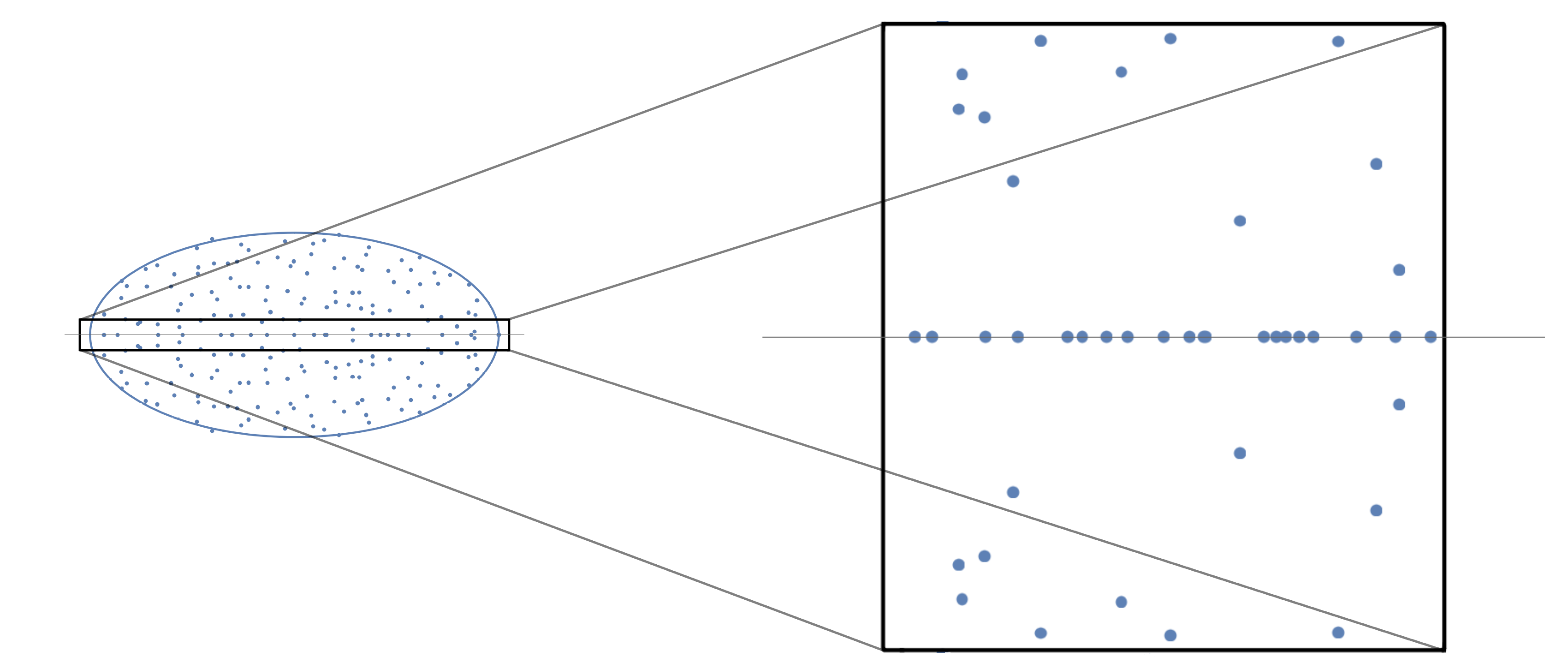}
    \caption{Eigenvalues of the elliptic GinOE, where $\tau=1/3$ and $N=200$.}
    \label{Fig_REGinibre}
\end{figure}

From now on, we shall focus on the case that the matrix dimension $N$ is even. 
The odd $N$ case can also be analysed but requires further non-trivial computations \cite{Si07,FM09}. 
Let us denote by $\{ \lambda_j \}_{j=1}^{\mathcal{N}_\tau}$ the real eigenvalues of the elliptic GinOE, where $\mathcal{N}_\tau$ is the number of real eigenvalues. 
Along the discussions above, we write 
\begin{equation} \label{Mp definition}
M_p \equiv M_{p,\tau} :=  \mathbb{E} \Big[ \sum_{j=1}^{ \mathcal{N}_\tau } \lambda_j^p \Big] 
\end{equation}
for the $p$-th spectral moment of real eigenvalues of the elliptic GinOE. 
In particular, the zero-th moment 
\begin{equation}
M_{0,\tau} = \mathbb{E} \mathcal{N}_\tau
\end{equation}
corresponds to the expected number of real eigenvalues. 
In contrast to the GUE or GOE cases, the evaluation of the spectral moments of the elliptic GinOE does not permit simple formulas and requires highly non-trivial analysis.
For instance, it was shown in \cite{FN08} that $M_{0,\tau} $ can be evaluated as
\begin{equation} \label{M0 expected number}
M_{0,\tau} = \Big(\frac{2}{\pi} \frac{1+\tau}{1-\tau} \Big)^{\frac12} \sum_{k=0}^{N/2-1} \frac{\Gamma(2k+\frac12)}{(2k)!} {}_2 F_1\Big( \frac12, \frac12 \, ; \, -2k+\frac12 \, ; \, -\frac{\tau}{1-\tau} \Big),
\end{equation}
where ${}_2 F_1$ is the hypergeometric function defined by 
\begin{equation} \label{2F1 Gauss series}
	{}_2 F_1(a,b;c;z):=\frac{\Gamma(c)}{\Gamma(a)\Gamma(b)} \sum_{s=0}^\infty \frac{\Gamma(a+s) \Gamma(b+s) }{ \Gamma(c+s) s! } z^s, \quad (|z|<1)
\end{equation}
and by the analytic continuation otherwise. 
The expression \eqref{M0 expected number} results from intricate computations as well as extensive simplifications, utilizing the skew-orthogonal polynomial formalism developed by Forrester and Nagao in \cite{FN07,FN08}.
For the reader's convenience, we provide a brief exposition of this formalism and the derivation of \eqref{M0 expected number} in Appendix~\ref{Appendix_integrable}.
For the extremal case when $\tau=0,1$, the formula \eqref{M0 expected number} reduces to 
\begin{equation}
M_{0,\tau} = \begin{cases}
   \displaystyle \sqrt{2} \sum_{k=0}^{N/2-1} \dfrac{(4k-1)!!}{ (4k)!! } &\textup{if }\tau=0,
    \smallskip 
    \\
    N &\textup{if }\tau=1. 
\end{cases}
\end{equation}
This formula for the GinOE case $\tau=0$ was first obtained in the work \cite{EKS94} of Edelman, Kostlan and Shub. 
On the other hand, for the GOE case $\tau=1$, this is trivial because GOE matrices are inherently symmetric. 
This can also be interpreted as the random variable $\mathcal{N}_\tau$ becoming deterministic in the Hermitian limit $\tau \uparrow 1$.
We mention that the large-$N$ asymptotic behaviours of $M_{0,\tau}$ have been extensively studied in recent works \cite{BKLL23,BMS23}.

\medskip 

We now present our main result. 
For this purpose, for $l=\{1,\dots,10\},$ let 
\begin{align}
\begin{split} \label{mathfrak A def}
\mathfrak{A}_l \equiv \mathfrak{A}_l(p;\tau)&:=  \frac{1}{( 1  + 6 \tau + 2 N (1-\tau^2) ) (4  + 17 \tau + 6 \tau^2 + 2 N(1- \tau^2) )^2} 
\\
&\quad \times \frac{ 1 }{ 
256(p-4) (p-3) (p-2) } \sum_{ k=0 }^{ \min\{ 10-l, 7 \} } (2p-k-2l+1)_{k+2l-3} \,\mathfrak{a}_{k,k+2l-3},
\end{split}
\end{align}  
where $(a)_n=a (a+1) \dots (a+n-1)$ is the Pochhammer symbol and $\mathfrak{a}_{l,k} \equiv \mathfrak{a}_{l,k}(\tau)$ is the coefficient of the $t^{k}$-term in the polynomial $A_l(t)$ (see \eqref{mathfrak a lk definition} below).
Here, $A_l(t)$'s are odd or even polynomials of degree $17-l$ that are given in \eqref{Ak definition Bj}. 

\begin{thm}[\textbf{Recursion formula of the elliptic GinOE}] \label{Thm_main}
Let $\tau \in [0,1].$
Then for any integer $p \ge 10$ and even integer $N \ge 2$, we have the 11-term recurrence relation
\begin{equation} \label{M recurrence main}
2(2p+1)(1-\tau)^6 \, M_{2p,\tau} = \sum_{l=1}^{10} \mathfrak{A}_l(p;\tau) \, M_{2p-2l,\tau},
\end{equation}
where $\mathfrak{A}_l$'s are given by \eqref{mathfrak A def}. 
\end{thm}

We note that while it is possible to write the coefficients $\mathfrak{A}_l$ explicitly, the resulting expressions are seemingly complicated, see Section~\ref{Section_polynomials}.  
Nevertheless, when considering the extremal cases $\tau=0,1$, we encounter dramatic simplifications: 
\begin{equation}
\mathfrak{A}_l=0, \qquad  \begin{cases}
\textup{if }\tau=0 \textup{ and } l \ge 3,
\smallskip 
\\
\textup{if }\tau=1 \textup{ and } l \le 5. 
\end{cases}
\end{equation}
This leads to the following corollary.

\begin{cor}[\textbf{Recursion formula of the GinOE and GOE}] \label{Cor_Recursion GinOE GOE}
For any even integer $N \ge 2$, we have the following.
\begin{itemize}
    \item[\textup{(i)}] \textup{\textbf{The GinOE case ($\tau=0$)}}. 
For any integer $p \ge 2$, we have the 3-term recurrence relation 
\begin{align}
\begin{split} \label{GOE recursion}
2(2p+1) M_{2p,0}  &= (2p-1) (6p+4N-5) M_{2p-2,0}  - (2p-3) (2p+N-4)(2p+2N-3) M_{2p-4,0}.
\end{split}
\end{align}
    \item[\textup{(ii)}] \textup{\textbf{The GOE case ($\tau=1$)}}.  For any integer $p \ge 4$, the recurrence relation \eqref{HZ formula for GOE} holds.   
\end{itemize}
\end{cor}

We mention that Corollary~\ref{Cor_Recursion GinOE GOE} (ii) reproduces the result \cite[Theorem 2]{Le09} of Ledoux.
On the other hand, the recurrence relation \eqref{GOE recursion} for real eigenvalues of the GinOE matrix has not been reported in the literature to the best of our knowledge. 
After the present work, the GinOE case has been further investigated in a recent literature \cite{BF24}. 
As previously mentioned, the recursion formulas can be used to study the counting statistics of various random matrix ensembles, cf. Remarks~\ref{Rem_Large N GinOE} and ~\ref{Rem_large N elliptic GinOE}.
In particular, the counting statistics of non-Hermitian random matrix ensembles have been the subject of recent active research, see e.g. \cite{ACCL22, Ch22, Ch23, ABES23, ABE23, FL22, Be23,GLX23,ACM23,SLMS21,SLMS22} and references therein.

The key ingredient for the proof of Theorem~\ref{Thm_main} is the moment generating function (the Laplace transform) 
\begin{equation} \label{M(t) definition}
M(t) \equiv M_\tau(t) := \mathbb{E} \Big[ \sum_{j=1}^{ \mathcal{N}_\tau } e^{t \,\lambda_j} \Big] 
\end{equation}
of real eigenvalues. As previously remarked, the recursion formula of the spectral moments can be derived from a linear differential equations for $M_\tau$, cf. \eqref{M tau derivatives}. 
The differential equation for $M_\tau$ is formulated in the following theorem, which is of independent interest.

\begin{thm}[\textbf{\textup{Differential equation for the MGF of the elliptic GinOE}}]\label{Thm_M ODE}
Let $\tau \in [0,1].$ Then for any even integer $N \ge 2$, the moment generating function $M_\tau(t)$ satisfies the seventh-order differential equation
\begin{equation} \label{M ODE main}
\sum_{k=0}^7 A_k(t) M^{(k)}_\tau(t)=0,
\end{equation}
where the polynomials $A_k$'s are given in \eqref{Ak definition Bj}. 
\end{thm}

The existence of such a linear differential equation for the moment generating function is already far from obvious.
Furthermore, it is crucial to highlight that the coefficients in \eqref{M ODE main} are polynomials. This, in turn, results in a finite-term recurrence relation for the spectral moments.


For the extremal cases, explicit evaluations of the polynomials $A_k$ are provided in Example~\ref{Examples_A tau 0 1}, where we once again observe significant simplifications. As a consequence, we have the following immediate corollary.

\begin{cor}[\textbf{Differential equation for the MGF of the GinOE and GOE}]
\label{Cor_ODE for GinOE and GOE}
For any even integer $N \ge 2$, we have the following.

\begin{itemize}
    \item[\textup{(i)}] \textup{\textbf{The GinOE case ($\tau=0$)}}. For $\tau=0$, we have 
\begin{equation} \label{M ODE tau0}
\sum_{k=1}^7 C_k(t) M^{(k)}_0(t)=0,
\end{equation}
where 
\begin{align}
\begin{split}
C_7(t) &= 2 \, t^4,
\qquad 
C_6(t) = -3 \, t^5+2 \, t^3,
\qquad 
C_5(t) = t^6-(4N+13)\, t^4-42\, t^2,
\\ 
C_4(t) &= (3N+5)\,t^5+4(N+10)\,t^3+120\,t,
\qquad  
C_3(t) = (N-1)(2N+3) \, t^4+36N \, t^2-120,
\\ 
C_2(t) &= -6(N+1)^2\, t^3-120(N+1)\, t, 
\qquad 
C_1(t) = 6(N+1)^2\, t^2+120(N+1). 
\end{split}
\end{align}
    \item[\textup{(ii)}] \textup{\textbf{The GOE case ($\tau=1$)}}. For $\tau=1$, we have 
\begin{equation}
\sum_{k=0}^4 C_k(t) M^{(k)}_1(t)=0,
\end{equation}
where 
\begin{align}
\begin{split}
C_4(t) &= t,
\qquad 
C_3(t) = 5, 
\qquad 
C_2(t) = -5t^3-(8N-4)t,
\\
C_1(t) &= -36\,t^2-20N+10,
\qquad 
C_0(t) =  4\,t^5+(20N-10)\,t^3+(16N^2-16N-44)t. 
\end{split}
\end{align}
\end{itemize}

\end{cor}

We stress that Corollary~\ref{Cor_ODE for GinOE and GOE} (ii) for the GOE was previously derived and used in \cite[Section 3]{Le09}. 
On the other hand, the differential equation \eqref{M ODE tau0} for the GinOE is new to our best knowledge.
However, the spectral moments of the GinOE real and complex eigenvalues combined were studied in \cite{FR09,SK09}, see also \cite{BF24}.   

\medskip 

Let us discuss Corollary~\ref{Cor_ODE for GinOE and GOE} (i) from the viewpoint of the large-$N$ limit of the GinOE.

\begin{rem}[\textbf{Large-$N$ expansion of the GinOE}] \label{Rem_Large N GinOE}
For the GinOE case $\tau=0$, we define the rescaled moment generating function
\begin{equation}
\widetilde{M}(t):= \frac{1}{\sqrt{N}} M\Big(\frac{t}{\sqrt{N}}\Big) = \int_\R e^{tx} \rho_N(x) \,dx, 
\end{equation}
where $\rho_N(x):=R_N(\sqrt{N}x)$ is the rescaled density, see \eqref{RN rescaled tau0} for its explicit formula. 
By using \eqref{M ODE tau0} and the change of variables, one can observe that $\widetilde{M}$ satisfies the differential equation
\begin{equation} 
\sum_{k=1}^7  \widetilde{C}_k(t)\, \widetilde{M}^{(k)}(t)=0, \qquad  \widetilde{C}_k(t) : = N^{k/2}\,C_k\Big(\frac{t}{\sqrt{N}}\Big) .
\end{equation}
Expanding $\widetilde{C}_k$'s for large-$N$, this equation can be rewritten as 
\begin{equation} \label{M mathfrak D 0}
\mathfrak{D}(t) \Big[\widetilde{M}(t)\Big]=0 , \qquad 
\mathfrak{D}(t):=\mathfrak{D}_0(t) + \frac{\mathfrak{D}_1(t)}{N} + \frac{\mathfrak{D}_1(t)}{N^2},
\end{equation}
where the linear differential operators $\mathfrak{D}_{k}$'s are given by 
\begin{align}
\begin{split}
\mathfrak{D}_{0}(t)&:= 2 \, t^4 \, \partial_t^7 + 2 \, t^3 \,\partial_t^6  -(4t^4+ 42 t^2) \,\partial_t^5 +(4t^3+120\,t) \,\partial_t^4
\\
&\quad+ (2t^4+36t^2-120) \,\partial_t^3  -(6t^3+120t)\,\partial_t^2 +(6t^2+120)\,\partial_t,
\end{split}
\\
\begin{split}
\mathfrak{D}_{1}(t)&:= -3 \, t^5 \,\partial_t^6-13\,t^4\,\partial_t^5 +(3\,t^5+40\,t^3) \,\partial_t^4 + t^4\,\partial_t^3 -(12t^3+120t ) \,\partial_t^2 + (12t^2+120) \,\partial_t,
\end{split}
\\
\begin{split}
\mathfrak{D}_{2}(t)&:= t^6\,\partial_t^5 +5t^2\,\partial_t^4  -3 \, t^4  \,\partial_t^3  -6t^3 \,\partial_t^2  +6t^2  \,\partial_t.
\end{split}
\end{align}


On the one hand, it is well known that the rescaled density $\rho_N(x)$ satisfies the expansion 
\begin{equation}
\rho_N(x) = \rho_{(0)}(x) + \rho_{(1/2)}(x) \frac{1}{\sqrt{N}}  + O\Big(\frac{1}{N}\Big) 
\end{equation} 
in the sense of distribution, where
\begin{equation}
\rho_{(0)}(x):=\frac{1}{\sqrt{2\pi}} \, \mathbbm{1}_{(-1,1)}(x), \qquad 
\rho_{(1/2)}(x):= \frac14\Big( \delta(x-1)+\delta(x+1) \Big) ,
\end{equation}
see e.g. \cite[pp. 33--34]{BF23}. 
Then it follows that 
\begin{equation}
\widetilde{M}(t)= \widetilde{M}_{(0)}(t)  +  \widetilde{M}_{(1/2)}(t) \frac{1}{\sqrt{N}}+O\Big(\frac{1}{N}\Big) ,
\end{equation}
where
\begin{equation} \label{M0 1/2 sinh cosh}
\widetilde{M}_{(0)}(t) :=\int_\R e^{tx} \rho_{(0)}(x) \,dx=  \sqrt{ \frac{2}{\pi} } \frac{ \sinh(t) }{t}, \qquad \widetilde{M}_{(1/2)}(t) :=\int_\R e^{tx} \rho_{(1/2)}(x) \,dx=  \frac{\cosh(t)}{2}. 
\end{equation}
Furthermore, by \eqref{M mathfrak D 0}, we have
\begin{align*}
\mathfrak{D}(t) \Big[\widetilde{M}(t)\Big]  = \Big[ \mathfrak{D}_0(t) + \frac{\mathfrak{D}_1(t)}{N} + \frac{\mathfrak{D}_1(t)}{N^2} \Big] \Big[\widetilde{M}_{(0)}(t)  +  \widetilde{M}_{(1/2)}(t) \frac{1}{\sqrt{N}}+O\Big(\frac{1}{N}\Big)   \Big] , 
\end{align*}
which in turn implies
\begin{equation}
\mathfrak{D}_0(t) \Big[ \widetilde{M}_{(0)}(t) \Big] =  \mathfrak{D}_0(t) \Big[ \widetilde{M}_{(1/2)}(t) \Big]=0.
\end{equation}
Since we have the explicit formulas \eqref{M0 1/2 sinh cosh}, one can directly check that these identities hold.

In the opposite direction, one can make use of \eqref{M0 1/2 sinh cosh} to find the factorisations
\begin{align*}
\begin{split}
\mathfrak{D}_{0}(t) & = 2(t^3\,\partial_t^3-6t^2\,\partial_t^2+15t\,\partial_t-15)  \circ(t\,\partial_t^2+4\,\partial_t-t)\circ(\pa_t^2-1)
\\
&= 2(t^3\,\partial_t^3-6t^2\,\partial_t^2+15t\,\partial_t-15)  \circ (\partial_t^2-1) \circ( t\,\partial_t^2+2\partial_t-t ).
\end{split}
\end{align*}
We mention that $\widetilde{M}_{(1/2)}(t)$ is annihilated by the the rightmost differential operator in the first line, while $\widetilde{M}_{(0)}(t)$ is annihilated by the the rightmost differential operator in the second line.
Furthermore, since the general solution of the differential equation
$$
(t^3\,\partial_t^3-6t^2\,\partial_t^2+15t\,\partial_t-15)f(t) =0 
$$
is of the form $f(t)=c_1t+c_2t^3+c_3t^5$, one can further factorise and obtain  
\begin{align}
\begin{split}
\mathfrak{D}_{0}(t) & = 2 (t \,\partial_t-a)\circ(t \,\partial_t-b)\circ(t \,\partial_t-c)  \circ(t\,\partial_t^2+4\,\partial_t-t)\circ(\pa_t^2-1)
\\
&= 2 (t \,\partial_t-a)\circ(t \,\partial_t-b)\circ(t \,\partial_t-c)  \circ (\partial_t^2-1) \circ( t\,\partial_t^2+2\partial_t-t ),
\end{split}
\end{align}
where $( a,b,c )$ is a permutation of $(1,3,5)$.
\end{rem}

\begin{rem}[\textbf{Large-$N$ expansion of the elliptic GinOE}] \label{Rem_large N elliptic GinOE}
Continuing the discussions from the previous remark, it is natural to consider the large-$N$ expansion of the elliptic GinOE for a general value of $\tau \in [0, 1]$.
Once again, we introduce the rescaled density $\rho_N(x) := R_N(\sqrt{N}x)$. 
In the study of the large-$N$ asymptotics of the elliptic Ginibre matrices, one needs to distinguish the following two different regimes.
\begin{itemize}
    \item  \textbf{Strong non-Hermiticity}. This is the case where $\tau$ is fixed in the interval $[0,1)$. 
    In this regime, it was shown in \cite[Section 6.2]{FN08} that 
    \begin{equation} \label{rhoN strong}
   \rho_N(x)= \frac{1}{\sqrt{2\pi(1-\tau^2)} } \,\mathbbm{1}_{(-1-\tau,1+\tau)}(x) +o(1). 
    \end{equation}
    Therefore, one can observe that 
    \begin{equation} \label{M2p leading strong}
    \int_\R x^{2p} \,\rho_N(x) \,dx \sim M_{2p}^{ \rm s }, \qquad M_{2p}^{ \rm s }:= \frac{1}{\sqrt{2\pi(1-\tau^2)} } \frac{ 2 }{ 2p+1 } (1+\tau)^{2p+1}.
    \end{equation}
    \item \textbf{Weak non-Hermiticity}. This is the case that $\tau=1-\frac{\alpha^2}{N}$ for a fixed $\alpha \in (0,\infty)$. 
    This regime was introduced by Fyodorov, Khoruzhenko and Sommers \cite{FKS97,FKS98} in the study of the critical transition of the elliptic GinUE statistics, see \cite{ACV18,AB23} for further references and physical applications.  
    In this regime, it was shown in \cite[Theorem 2.3]{BKLL23} that 
\begin{equation} \label{rhoN weak}
\rho_N(x)=  \frac{ \sqrt{N} }{2\alpha\sqrt{\pi}} \erf\Big(\frac{\alpha}{2}\sqrt{4-x^2}\Big)\,\mathbbm{1}_{(-2,2)}(x) +o(\sqrt{N}). 
\end{equation}
We mention that this formula was previously proposed by Efetov in \cite{Efe97}. 
Note that the leading order density in \eqref{rhoN weak} interpolates the uniform density in \eqref{rhoN strong} with the Wigner's semi-circle law \eqref{semi circle}. 
Using \eqref{rhoN weak}, one can show that  
\begin{equation} \label{M 2p w}
 \int_\R x^{2p}\, \rho_N(x) \,dx \sim  \sqrt{N} \,M_{2p}^{ \rm w }, \qquad  M_{2p}^{ \rm w }:= \frac{ 1 }{ \pi } \,  \sum_{n=0}^{\infty}\frac{ 2^{2p+1}  }{ 2n+1 }      \frac{ \Gamma(n+\frac32) \Gamma(p+\frac12) }{ n!\, (n+p+1)! } \, (-\alpha^2)^{n} .
\end{equation}
We stress that in the Hermitian limit $\alpha \to 0$, the moment $M_{2p}^{ \rm w }$ recovers the Catalan number in \eqref{Catalan}, i.e. 
\begin{equation}
\lim_{\alpha \to 0} M_{2p}^{ \rm w } = c(0;p). 
\end{equation}
In the opposite limit $\alpha \to \infty$, one can also match $M_{2p}^{ \rm w }$ with $M_{2p}^{ \rm s }$, see \cite{BKLL23} for a similar discussion. 
Let us also mention that $M_{2p}^{ \rm w }$ can be written in terms of the the modified Bessel function of the first kind 
\begin{equation}\label{I nu}
	I_\nu(z):=\sum_{k=0}^\infty \frac{(z/2)^{2k+\nu}}{ k! \, \Gamma(\nu+k+1) },
\end{equation}
see e.g. \cite[Chapter 10]{NIST}. To be more precise, it is of the form 
\begin{equation} \label{M 2p w I0I1}
M_{2p}^{ \rm w } =  \Big[ r_{0,2p} \, I_0\Big( \frac{\alpha^2}{2} \Big)+ r_{1,2p} \, I_1\Big( \frac{\alpha^2}{2} \Big) \Big] e^{-\frac{\alpha^2}{2}} ,
\end{equation}
where $r_{0,2p}$ and $r_{1,2p}$ are some rational functions of $\alpha$.
For instance, the first few values of $r_{0,2p}$ and $r_{1,2p}$ are given by
\begin{align}
& r_{0,0}=1, \qquad r_{0,2}= \frac{8(2\alpha^4-\alpha^2)}{5\alpha^4}, \qquad r_{0,4}= \frac{64 ( 4 \alpha^8- 6 \alpha^6 + 15 \alpha^4-24\alpha^2    ) }{9\alpha^8},
\\
&  r_{1,0}=1, \qquad r_{1,2}= \frac{8(2\alpha^4-3\alpha^2+4)}{5\alpha^4}, \qquad r_{1,4}=  \frac{64( 4 \alpha^8 - 10 \alpha^6  + 27 \alpha^4 - 60 \alpha^2+96 )}{9\alpha^8}.
\end{align}
In Appendix~\ref{Appendix_integrable}, we provide the derivation of \eqref{M 2p w} and \eqref{M 2p w I0I1}. 
\end{itemize}
We have discussed the leading-order asymptotics of the spectral moments for both the strong and weak non-Hermiticity regimes. 
The more precise asymptotic expansions require extending the main findings in \cite{FN08,BKLL23}, which exceed the scope of this paper. 
We hope to come back to this topic in future work and find further applications of Theorems~\ref{Thm_main} and ~\ref{Thm_M ODE}.
\end{rem}

\subsection*{Plan of the paper}
The rest of this paper is organised as follows.
In Section~\ref{Section_polynomials}, we present explicit formulas for the polynomials $A_k$ used in Theorems~\ref{Thm_main} and ~\ref{Thm_M ODE}.
Section~\ref{Section_proof main} provides an outline of the proof and culminates in the proof of the main theorems.
Section~\ref{Section_ODEs} is devoted to the remaining proofs, especially the derivation of several differential equations used to prove Theorem~\ref{Thm_M ODE}. 
Appendix~\ref{Appendix_integrable} provides a summary of the integrable structure of the elliptic GinOE due to Forrester and Nagao \cite{FN08} and also provides derivations of the formulas in Remark~\ref{Rem_large N elliptic GinOE}.

\subsection*{Acknowledgements} 
The author gratefully thanks Peter J. Forrester for several helpful comments and discussions, notably for directing the author's attention to the discussion in Remark~\ref{Rem_Large N GinOE}. 
Thanks are extended to Nick Simm for his interest and valuable suggestions on Remark~\ref{Rem_large N elliptic GinOE}.
The author also thanks Jaeseong Oh for stimulating conversations.


\section{Polynomial coefficients} \label{Section_polynomials}

In this section, we introduce the polynomials $A_k$ used in Theorems~\ref{Thm_main} and ~\ref{Thm_M ODE}. 
These polynomials are constructed from the following basic polynomials:
\begin{align}
\begin{split} \label{a(t) definition}
a(t) &:= -2\tau^2(1+\tau)^3(1+2\tau) \,t^6 + 4\tau(1+\tau)^2 \Big(1+8\tau-5 \tau (1-\tau^2)N \Big)\, t^4
\\
&\quad +8(1-\tau^2) \Big( 1-3\tau-30\tau^2+(1-\tau^2)(3+4\tau+4\tau^2)N+2(1-\tau^2)^2 N^2 \Big) \,t^2
\\
&\quad +32(1-\tau)^2 \Big(1+6\tau+2(1-\tau^2)N \Big);
\end{split}
\\
\begin{split} \label{b(t) definition}
b(t)  &:= -2 \tau (1 + \tau)^2 ( 1 - \tau - 6 \tau^2) \, t^5+ 4 ( 1 - \tau^2) \Big( ( 1 - 2 \tau) (1 + 7 \tau) + 2  ( 1 - \tau)^3 (1 + \tau) N   \Big) \, t^3
\\
&\quad - 32 ( 1 - \tau)^2 \Big( 1 + 6 \tau + 2  ( 1 - \tau^2 ) N \Big) \, t;
\end{split}
\\
\begin{split}  \label{c(t) definition}
c(t)  &:=  4 \tau (1 - \tau^2)  (1 + 2 \tau) \,t^4
 - 8 (1 - \tau)^2 \Big( 1 + 6 \tau + 2 N ( 1 - \tau^2)\Big)\, t^2;
\end{split}
\\
 \label{d(t) definition}
d(t) & :=  \tau (1 + \tau) (2 + \tau) (1 + 2 \tau) \,t^5 -  2 (1 - \tau) \Big(  4+17\tau+6\tau^2+2N(1-\tau^2)  \Big)  t^3.
\end{align}
In Lemma~\ref{Lem_u and V}, these polynomials appear as linear coefficients when expressing the elliptic GinUE part of the moment generating function $u(t)$ in terms of the function $V(t)$, cf. \eqref{u v definition} and \eqref{V definition}.

\subsection{The polynomials $B_k$} \label{Subsection beta jk}

First, we introduce the polynomials $B_k$ $(k=0,\dots,4)$, which serve as the linear coefficients in the differential equation for the function $V(t)$, see \eqref{V ODE}. 
Before providing their definitions, let us mention some of their basic properties:
\begin{itemize}
    \item In general, the degree of polynomials $B_k$ is $14-k$;
    \smallskip
    \item The polynomial $B_k$ satisfy $B_k(t) = O(t^k)$ as $t\to 0$; 
    \smallskip 
    \item The polynomial $B_k$ is even if $k$ is even and odd if $k$ is odd.
\end{itemize}
To be explicit, the polynomials $B_k$ are defined as 
\begin{equation} \label{Bk definition}
B_k(t)= \frac{1}{t^4} \Big( \beta_{k,0}(t)  d(t)^2 + \beta_{k,1}(t)   d(t) d'(t) + \beta_{k,2}(t) d(t) d''(t) + \beta_{k,3}(t) d'(t)^2   \Big),
\end{equation}
where $d(t)$ is given in \eqref{d(t) definition}. 
Here, the polynomials $\beta_{k,j}$'s are given in terms of $a(t), b(t)$ and $c(t)$ in \eqref{a(t) definition}, \eqref{b(t) definition} and \eqref{c(t) definition} as follows.

\begin{itemize}
    \item $k=4$:
    \begin{align*}
    \beta_{4,0}(t)=  4(1-\tau) \,c(t), \qquad \beta_{4,1}(t)=\beta_{4,2}(t)=\beta_{4,3}(t)= 0. 
    \end{align*}
    \item $k=3$:
    \begin{align*}
    \begin{split}
    &\beta_{3,0}(t)= 4(1-\tau) \, \Big( b(t)+2 c'(t) \Big) +  \Big( 2\tau(1+\tau)(4-\tau) \,t^2 -4(1-\tau) \Big) \frac{ c(t) }{ t },
    \\
    &\beta_{3,1}(t) =  -  8(1-\tau)\, c(t), \qquad \beta_{3,2}(t)=\beta_{3,3}(t)=0.
    \end{split}
    \end{align*}
    \item $k=2$: 
    \begin{align*}
    \begin{split}
    &\beta_{2,0}(t) = 4(1-\tau)\, \Big(a(t) + 2 b'(t) + c''(t) \Big) 
    + \Big( 2\tau(1+\tau)(4-\tau) \,t^2 - 4(1-\tau) \Big)  \frac{ b(t)+ c'(t)}{ t } +  3\tau^2  (1+\tau)^2 \,t^2\, c(t) ,
    \\
    &\beta_{2,1}(t) = - 8(1-\tau) \,\Big( b(t)+c'(t)  \Big)- \Big( 2\tau(1+\tau)(4-\tau) \,t^2 -4(1-\tau) \Big) \frac{c(t)}{t} ,  
    \\
    &\beta_{2,2}(t) = -4(1-\tau) c(t), \qquad \beta_{2,3}(t)=  8(1-\tau) c(t).
    \end{split} 
    \end{align*}
    \item $k=1$:
    \begin{align*}
    \begin{split}
    & \beta_{1,0}(t) =  4(1-\tau)\Big(2 a'(t) + b''(t)\Big) 
    +\Big( 2\tau(1+\tau)(4-\tau) \,t^2 -4(1-\tau) \Big)  \frac{a(t)+b'(t)}{t} +  3\tau^2  (1+\tau)^2 \,t^2 b(t),
    \\
    &\beta_{1,1}(t) = -  8(1-\tau) \Big(a(t)+b'(t)\Big) -  \Big( 2\tau(1+\tau)(4-\tau) \,t^2 -4(1-\tau) \Big) \frac{ b(t) }{ t }, 
    \\
    &\beta_{1,2}(t) = -4(1-\tau)\, b(t), \qquad \beta_{1,3}(t)=8(1-\tau)  \, b(t) .
    \end{split}
    \end{align*}
    \item $k=0$: 
    \begin{align*}
    \begin{split}
    & \beta_{0,0}(t) =  4(1-\tau) \,a''(t)+ \Big( 2\tau(1+\tau)(4-\tau) \,t^2 -4(1-\tau) \Big) \frac{a'(t)}{t} 
 + 3\tau^2  (1+\tau)^2 \,t^2\, a(t) -4 N \tau^2(1+\tau)^5\,t \,d(t) ,
    \\
    & \beta_{0,1}(t) = -  8(1-\tau)a'(t) -\Big( 2\tau(1+\tau)(4-\tau) \,t^2 -4(1-\tau) \Big)  \frac{a(t)}{t} , 
    \\
    & \beta_{0,2}(t) = - 4(1-\tau)\,a(t), \qquad \beta_{0,3}(t)= 8(1-\tau) a(t) .
    \end{split}
    \end{align*}
\end{itemize}

Before explaining the use of the polynomials $B_k$, let us consider their extremal cases $\tau=0,1$, which highlight how they simplify in these situations.

\begin{ex} \label{Example_Bk tau01}
For the extremal cases $\tau=0,1$, we have the following. 

\begin{itemize}
    \item The GinOE case ($\tau=0$). In this case, it follows from 
    \begin{align}
    \begin{split}
a(t) &= 8 (2N+1)(N+1) \,t^2 +32 (2N+1),
\qquad 
b(t)   =  4  (2N+1)\, t^3- 32  (2N+1)\, t,
\\
c(t)  & =  - 8 (2N+1)\, t^2,
\qquad 
d(t) = -4(N+2) \,t^3  
    \end{split} 
\end{align}
that 
\begin{align}
\begin{split}
B_4(t)&= 256 (2N+1)(N+2)^2  \Big(-2t^4\Big), 
\qquad 
B_3(t) = 256 (2N+1) (N+2)^2  \Big( t^5-2t^3 \Big),
\\
B_2(t) &= 256 (2N+1)(N+2)^2  \Big(2 (N+1) t^4+42 t^2   \Big),
\\
B_1(t) &= 256 (2N+1)(N+2)^2  \Big( -6(N+1) t^3 -120 t   \Big),
\\
B_0(t) &= 256 (2N+1)(N+2)^2  \Big(6(N+1)t^2+120 \Big).
\end{split}
\end{align}
    \item The GOE case ($\tau=1$). In this case, it follows from 
    \begin{equation}
a(t) = -48 \,t^6+ 144\,t^4, \qquad b(t)   =  48\,t^5, \qquad c(t)   = 0, \qquad d(t) = 18t^5 
\end{equation}
that
\begin{align}
\begin{split}
B_4(t) &= B_3(t) = 0, \qquad B_2(t) = 186624 \,t^{12} ,
\\
B_1(t) &= 559872 \, t^{11} ,
\qquad
B_0(t) = 186624 \, t^{10} ( -t^4-2(2N-1)t^2-3 ). 
\end{split}
\end{align} 
\end{itemize}
\end{ex}

\subsection{The polynomials $A_k$} \label{Subsection alpha jk}

We are now ready to introduce the polynomials $A_k$ $(k=0,\dots, 7)$ used in Theorems~\ref{Thm_main} and ~\ref{Thm_M ODE}. 
The polynomials $A_k$'s are defined in terms of $B_k$'s in the previous subsection as  
\begin{equation} \label{Ak definition Bj}
A_k(t) = \sum_{j=0}^4 \alpha_{k,j}(t)\,B_j(t),
\end{equation}
where $\alpha_{k,j}$'s are given as follows. 

\begin{itemize}
    \item $k=7$: 
    \begin{align*}
\alpha_{7,4}(t)=  \frac{1-\tau}{2}, \qquad \alpha_{7,3}(t)=\alpha_{7,2}(t)=\alpha_{7,1}(t)=\alpha_{7,0}(t)=0.
\end{align*}
   \item $k=6$:  
\begin{align*}
\begin{split}
&\alpha_{6,4}(t)=  -\frac{(1+\tau)(\tau^2-3\tau+1)}{2} \,t, 
\qquad \alpha_{6,3}(t)=  \frac{1-\tau}{2}, \qquad  \alpha_{6,2}(t)=\alpha_{6,1}(t)=\alpha_{6,0}(t)=0.
\end{split}
\end{align*}
  \item $k=5$: 
  \begin{align*}
  \begin{split}
&\alpha_{5,4}(t) = - \frac{1+\tau}{2}  \Big( 2\tau(1-\tau^2) \, t^2 +(1-\tau^2)N+5-15\tau+6\tau^2 \Big) , 
\\
& \alpha_{5,3}(t) =  -\frac{(1+\tau)(\tau^2-3\tau+1)}{2} \,t, \qquad \alpha_{5,2}= \frac{1-\tau}{2}, \qquad 
\alpha_{5,1}(t)=\alpha_{5,0}(t)=0. 
  \end{split}
\end{align*}
\item $k=4$: 
\begin{align*}
\begin{split}
&\alpha_{4,4}(t)=  -\frac{\tau (1+\tau)^2}{2}    \Big( \tau(1+\tau) \,t^2+(1+\tau)N+18-19\tau \Big)   \,t,
\\
& \alpha_{4,3}(t) = - \frac{1+\tau}{2}   \Big( 2\tau(1-\tau^2) \, t^2 +(1-\tau^2)N+ 4-12\tau+5\tau^2 \Big),
\\
& \alpha_{4,2}(t) =  -\frac{(1+\tau)(\tau^2-3\tau+1)}{2} \,t,
\qquad \alpha_{4,1}(t)= \frac{1-\tau}{2}, \qquad \alpha_{4,0}(t)=0 .
\end{split}
\end{align*}
\item $k=3$:
\begin{align*}
\begin{split}
&\alpha_{3,4}(t)= -2\tau(1+\tau)^2     \Big( 3\tau(1+\tau) \,t^2+(1+\tau)N+8-9\tau \Big) ,
\\
&\alpha_{3,3}(t)=  -\frac{\tau (1+\tau)^2}{2}    \Big( \tau(1+\tau) \,t^2+(1+\tau)N+14-15\tau \Big)   \,t ,  
\\
&\alpha_{3,2}(t)=  - \frac{1+\tau}{2}  \Big( 2\tau(1-\tau^2) \, t^2 +(1-\tau^2)N+ 3-9\tau+4\tau^2 \Big)  ,
\\
&\alpha_{3,1}(t) -\frac{(1+\tau)(\tau^2-3\tau+1)}{2} \,t\, B_1(t) , \qquad \alpha_{3,0}(t)= \frac{1-\tau}{2} B_0(t). 
\end{split}
\end{align*}
\item $k=2$:
\begin{align*}
\begin{split}
&\alpha_{2,4}(t) = -18\tau^2(1+\tau)^3   \, t, 
\qquad \alpha_{2,3}(t)=     -\frac{3\tau(1+\tau)^2}{2}  \Big( 3\tau(1+\tau) \,t^2+(1+\tau)N+6-7\tau \Big) ,
\\
&\alpha_{2,2}(t)= -\frac{\tau (1+\tau)^2}{2}  \Big( \tau(1+\tau) \,t^2+(1+\tau)N+10-11\tau \Big)   \,t ,
\\
&\alpha_{2,1}(t)= - \frac{1+\tau}{2} \Big( 2\tau(1-\tau^2) \, t^2 +(1-\tau^2)N+2-6\tau+3\tau^2 \Big) ,
\qquad  \alpha_{2,0}(t)-\frac{(1+\tau)(\tau^2-3\tau+1)}{2} \,t. 
\end{split}
\end{align*}
\item $k=1$:
\begin{align*}
\begin{split}
&\alpha_{1,4}(t)= -12\tau^2(1+\tau)^3,\qquad \alpha_{1,3}(t)= -9\tau^2(1+\tau)^3   \, t, 
\\
&\alpha_{1,2}(t)=-\tau(1+\tau)^2   \Big( 3\tau(1+\tau) \,t^2+(1+\tau)N+4-5\tau \Big) ,
\\
&\alpha_{1,1}(t) = -\frac{\tau(1+\tau)^2}{2}  \Big( \tau(1+\tau) \,t^2+(1+\tau)N+6-7\tau \Big) \,t ,
\\
&\alpha_{1,0}(t) = - \frac{1-\tau^2}{2} \Big( 2\tau(1+\tau) \, t^2 +(1+\tau)N+1-2\tau \Big) .
\end{split}
\end{align*}
\item $k=0$:
\begin{align*}
\begin{split}
& \alpha_{0,4}(t)=0, \qquad \alpha_{0,3}(t)= -3\tau^2(1+\tau)^3, 
\qquad \alpha_{0,2}(t)= -3\tau^2(1+\tau)^3 \,t ,
\\
&\alpha_{0,1}(t)= -\frac{\tau(1+\tau)^2}{2} \Big( 3\tau(1+\tau) \,t^2+(1+\tau)N+2-3\tau \Big), 
\\
&\alpha_{0,0}(t)=-\frac{\tau(1+\tau)^2}{2} \Big( \tau(1+\tau) \,t^2+(1+\tau)N+2-3\tau \Big) \,t.
\end{split}
\end{align*} 
\end{itemize}

\begin{rem}
It is straightforward to check that the polynomials $A_k$'s are of the form
\begin{align}
\begin{split} \label{mathfrak a lk definition}
A_7(t) &= \mathfrak{a}_{7,10} \, t^{10}+\dots+ \mathfrak{a}_{7,4} \, t^4,
\qquad 
A_6(t) = \mathfrak{a}_{6,11} \, t^{11}+\dots+ \mathfrak{a}_{6,3} \, t^3,
\\
A_5(t) &= \mathfrak{a}_{5,12} \, t^{12}+\dots+ \mathfrak{a}_{5,2} \, t^2,
\qquad 
A_4(t) = \mathfrak{a}_{4,13} \, t^{13}+\dots+ \mathfrak{a}_{4,1} \, t,
\\
A_3(t) &= \mathfrak{a}_{3,14} \, t^{14}+\dots+ \mathfrak{a}_{3,0} ,
\qquad \quad 
A_2(t) = \mathfrak{a}_{2,15} \, t^{15}+\dots+ \mathfrak{a}_{2,1} \, t,
\\
A_1(t)  &= \mathfrak{a}_{1,16} \, t^{16}+\dots+ \mathfrak{a}_{1,0},
\qquad \quad 
A_0(t) =  \mathfrak{a}_{0,17} \, t^{17}+\dots+ \mathfrak{a}_{0,7} \, t^7.
\end{split}
\end{align}
We also mention that $\mathfrak{a}_{0,2l-3} =0$ for $l = \{1,2,3,4\}$. 
Thus, in this case one can write  
\begin{equation}
\sum_{ k=1 }^{ \min\{ 10-l, 7 \} } \mathfrak{a}_{k,k+2l-3} (2p-k-2l+1)_{k+2l-3} 
\end{equation}
in the inner summation of \eqref{mathfrak A def}. 
\end{rem}

As before, let us discuss the extremal cases $\tau=0,1$, both of which exhibit notable simplifications.

\begin{ex} \label{Examples_A tau 0 1}
We have the following. 
\begin{itemize}
    \item The GinOE case ($\tau=0$). In this case, we have 
\begin{align}
\begin{split} 
A_7(t) &= \tfrac{1}{2} B_4(t), \qquad A_6(t) =  -\tfrac{1}{2} \,t\, B_4(t) + \tfrac{1}{2} B_3(t), 
\\
A_5(t) &= - \tfrac{1}{2}   (N+5)   B_4(t)   -\tfrac12 \,t \, B_3(t) + \tfrac{1}{2} B_2(t),
\\
A_4(t) &=  - \tfrac{1}{2}  (N+4) B_3(t) -\tfrac12 \,t\, B_2(t) + \tfrac{1}{2} B_1(t), 
\\
A_3(t) &=  - \tfrac{1}{2}  (N+3)  B_2(t) -\tfrac{1}{2} \,t\, B_1(t) + \tfrac12 B_0(t),
\\
A_2(t) &=  - \tfrac{1}{2} (N+2) B_1(t) -\tfrac{1}{2} \,t\, B_0(t),\qquad
A_1(t) =  - \tfrac{1}{2} (N+1) B_0(t), \qquad A_0(t) = 0.
\end{split}
\end{align}
    \item The GOE case ($\tau=1$). In this case, we have 
\begin{align}
    \begin{split}
A_7(t) &=A_6(t)= A_5(t) = 0,
\qquad
A_4(t) =  t\, B_2(t) ,
\qquad 
A_3(t) = 2  B_2(t) +t\, B_1(t) ,
\\
A_2(t) &=  -2 ( 2 \,t^2+2N-1 )   \,t\,  B_2(t)  + B_1(t)+ t\, B_0(t),
\\
A_1(t) &=  -2  ( 12\,t^2+4N-2 )  B_2(t)  -2 ( 2\,t^2+2N-1 ) \,t \, B_1(t), 
\\
A_0(t) &= -24 \,t \, B_2(t)  -2 ( 6 \,t^2+2N-1 ) B_1(t) 
 -2 ( 2 \,t^2+2N-1 ) \,t\, B_0(t) .
    \end{split}
\end{align}
\end{itemize}
\end{ex}

We mention that Examples~\ref{Example_Bk tau01} and ~\ref{Examples_A tau 0 1} immediately imply Corollaries~\ref{Cor_Recursion GinOE GOE} and ~\ref{Cor_ODE for GinOE and GOE}.


\section{Proof of main results} \label{Section_proof main}

In this section, we provide an outline and the proof of main theorems, postponing the remaining proofs of some key ingredients (Propositions~\ref{Prop_u ODE}, ~\ref{Prop_ODE M u} and Lemma~\ref{Lem_u and V}) to the next section.

The starting point of the proof is the integrable structure and skew-orthogonal polynomial representation of the $1$-point function $R_N$ of real eigenvalues. 
This is defined by its characteristic property
\begin{equation}
\mathbb{E} \bigg[\sum_{ j=1 }^{ \mathcal{N}_{\tau}  }  f( x_j ) \bigg] = \int_{ \mathbb{R} } f(x) \, R_N(x)\,dx ,
\end{equation}
where $f$ is a given test function. 
Note that the moment generating function \eqref{M(t) definition} and the spectral moments \eqref{Mp definition} can be written in terms of the $1$-point function $R_N$ as 
\begin{equation} \label{M Mp definition}
M(t):= \int_\R e^{tx} R_N(x)\, dx, 
\qquad M_p:= \int_\R x^p \,R_N(x)\,dx. 
\end{equation}

By using the skew-orthogonal polynomial formalism, the $1$-point function $R_N$ can be written in terms of the Hermite polynomials
\begin{equation}
H_k(x):= (-1)^k e^{x^2} \frac{d^n}{dx^n} e^{-x^2}.
\end{equation}
To be more precise, let us formulate \cite[Eq. (6.11)]{FN08} in the following lemma.

\begin{lem}[\textbf{Expression of the $1$-point function}] \label{Lem_RN Hermite}
For any even integer $N \ge 2$ and $\tau \in [0,1]$, we have
\begin{equation} \label{density eGinOE}
R_N(x)= R_N^{(1)}(x)+R_N^{(2)}(x),
\end{equation}
where 
\begin{align}
	R_{N}^{(1)}(x)&:=\frac{ e^{ -\frac{ x^2 }{ 1+\tau }  }  }{ \sqrt{2\pi} }  \sum_{k=0}^{N-2} \frac{(\tau/2)^k}{k!} H_k\Big( \frac{x}{ \sqrt{2\tau} } \Big)^2,  \label{RN1 def}
	\\
	\begin{split}
		R_{N}^{(2)}(x)&:=  \frac{(\tau/2)^{N-\frac32} }{ 1+\tau } \frac{ 1 }{ (N-2)! }  \frac{ e^{ -\frac{x^2 }{2(1+\tau)} } }{\sqrt{2\pi}}  H_{N-1}\Big( \frac{x}{ \sqrt{2\tau} } \Big) 
	 \int_{0}^{x} e^{ -\frac{ u^2 }{2(1+\tau)}  } H_{N-2} \Big( \frac{u}{ \sqrt{2\tau} } \Big)\,du.  \label{RN2 def}
	\end{split}
\end{align}
\end{lem}

See Appendix~\ref{Appendix_integrable} for further details and discussions on Lemma~\ref{Lem_RN Hermite}. 
We stress that $R_N^{(1)}$ indeed corresponds to the density of the elliptic GinUE (with $N \mapsto N-1$), restricted on the real axis, see \cite[Section 2.3]{BF22}.

\begin{ex}
For the extremal cases, we have the following.
\begin{itemize}
\item The GinOE case ($\tau=0$). It follows from 
\begin{equation} \label{Hermite asymp}
\Big( \frac{\tau}{2} \Big)^{k/2} H_k\Big( \frac{x}{ \sqrt{2\tau} } \Big) \to x^k, \qquad \tau \to 0
\end{equation}
that for $\tau=0$, we have
\begin{align}
	R_{N}^{(1)}(x)&:=\frac{ e^{ - x^2  }  }{ \sqrt{2\pi} }  \sum_{k=0}^{N-2} \frac{ x^{2k} }{k!}= \frac{1}{\sqrt{2\pi}} \frac{1}{(N-2)! } \Gamma(N-1,x^2) ,  \label{RN1 def tau0}
	\\
	\begin{split}
		R_{N}^{(2)}(x)&:= \frac{ 1 }{ (N-2)! }  \frac{ e^{ -\frac{x^2 }{2} } }{\sqrt{2\pi}} x^{N-1} 
	 \int_{0}^{x} e^{ -\frac{ u^2 }{2}  } u^{N-2}\,du=  \frac{ 2^{(N-3)/2} }{ (N-2)! }  \frac{ e^{ -\frac{x^2 }{2} } }{\sqrt{2\pi}} |x|^{N-1} \gamma\Big( \frac{N-1}{2}, \frac{x^2}{2} \Big).
  \label{RN2 def tau0}
	\end{split}
\end{align}
Here,
$$
\gamma(a,z) := \int_0^z e^{a-1} e^{-t}\,dt, \qquad \Gamma(a,z) :=  \int_z^\infty e^{a-1} e^{-t}\,dt =\Gamma(a)- \gamma(a,z)
$$
are lower and upper incomplete gamma functions.
Note that by \eqref{RN1 def tau0} and \eqref{RN2 def tau0}, the rescaled density $x \mapsto \sqrt{N}x$ is given by 
\begin{equation} \label{RN rescaled tau0}
R_N(\sqrt{N}x) = \frac{1}{\sqrt{2\pi}} \frac{1}{(N-2)! } \Gamma(N-1,Nx^2) +  \frac{  2^{(N-3)/2} }{ (N-2)! }  \frac{ e^{ -\frac{N x^2 }{2} } }{\sqrt{2\pi}} |\sqrt{N}x|^{N-1} \gamma\Big( \frac{N-1}{2}, \frac{Nx^2}{2} \Big) . 
\end{equation}
This formula appears in \cite[Corollary 4.3]{EKS94}. See \cite{ABES23} for a recent work on the use of this formula in the context of the counting statistics. 
\smallskip 
\item The GOE case ($\tau=1$). In this case, it is immediate to see that 
\begin{align}
\begin{split} \label{GOE density v1}
	R_{N}(x)& =\frac{ e^{ -\frac{ x^2 }{ 2 }  }  }{ \sqrt{2\pi} }  \sum_{k=0}^{N-2} \frac{1}{2^k\,k!} H_k\Big( \frac{x}{ \sqrt{2} } \Big)^2 
 + \frac{ 1 }{ 2^{N-1/2 } } \frac{ 1 }{ (N-2)! }  \frac{ e^{ -\frac{x^2 }{4} } }{\sqrt{2\pi}}  H_{N-1}\Big( \frac{x}{ \sqrt{2} } \Big) 
	 \int_{0}^{x} e^{ -\frac{ u^2 }{4}  } H_{N-2} \Big( \frac{u}{ \sqrt{2} } \Big)\,du.  
	\end{split}
\end{align}
This is the classical GOE density, see e.g. \cite[Section 6.4]{Fo10}. Notice here that the first line of \eqref{GOE density v1} coincides with the density of the GUE of size $N-1$.
We also stress that \eqref{GOE density v1} can be rewritten as  
\begin{align}
\begin{split} \label{GOE density v2}
	R_{N}(x) & =\frac{ e^{ -\frac{ x^2 }{ 2 }  }  }{ \sqrt{2\pi} }  \sum_{k=0}^{N-1} \frac{1}{2^k\,k!} H_k\Big( \frac{x}{ \sqrt{2} } \Big)^2 
 \\
 &\quad + \frac{ 1 }{ 2^{N+1/2 } } \frac{ 1 }{ (N-1)! }  \frac{ e^{ -\frac{x^2 }{4} } }{\sqrt{2\pi}}  H_{N-1}\Big( \frac{x}{ \sqrt{2} } \Big) 
	\bigg( \sqrt{\pi} \, 2^{N/2} (N-1)!!- \int_{x}^\infty e^{ -\frac{ u^2 }{4}  } H_{N-2} \Big( \frac{u}{ \sqrt{2} } \Big)\,du    \bigg).  
	\end{split}
\end{align}
We mention that the use of the formula \eqref{GOE density v2} was made in \cite{Le09}. 
Note also that the first line of \eqref{GOE density v2} corresponds to the density of the GUE of size $N$. 
We refer to \cite{AFNM00,Wi99} for a comprehensive discussion on such a relation between the correlation functions of unitary and orthogonal ensembles.
\end{itemize}
\end{ex}

We shall analyse the integrals of $R_N^{(1)}$ and $R_N^{(2)}$ separately.
For this purpose, let 
\begin{equation} \label{u v definition}
u(t):= \int_\R e^{ tx } \, R_N^{(1)}(x)\,dx, \qquad v(t):= \int_\R e^{ tx } \, R_N^{(2)}(x)\,dx. 
\end{equation}
Note that by \eqref{M Mp definition} and \eqref{density eGinOE}, we have 
\begin{equation} \label{M u v}
M(t) = \int_\R e^{tx} R_N(x)\, dx= u(t)+v(t). 
\end{equation}

We first derive alternative representations of $u(t),$ which allow us to perform further analysis.

\begin{prop}[\textbf{Representations of the MGF of the elliptic GinUE part}]\label{Prop_u v int rep}
For any even integer $N \ge 2$ and $\tau \in [0,1]$, we have the integral representation
\begin{equation} \label{u integral rep}
 u(t) = \frac{2 }{ 1+\tau }   \frac{ (\tau/2)^{N-\frac32 } }{(N-2)!}  \int_\R  \frac{e^{tx}}{t}\,  H_{N-2}\Big(\frac{x}{\sqrt{2\tau}} \Big) H_{N-1}\Big(\frac{x}{ \sqrt{2\tau} }\Big)  \frac{ e^{ -\frac{ x^2 }{ 1+\tau }  }  }{ \sqrt{2\pi} }  \, dx. 
\end{equation}
Furthermore, it can also be written as 
\begin{align}
\begin{split} \label{u discrete rep}
 u(t) = \sqrt{\frac{2 }{ 1+\tau }}   e^{ \frac{t^2(1+\tau)}{4} } 
\sum_{k=0}^{ N-2 } \frac{1 }{ k! } \binom{N-1}{k+1}  \,\tau^{N-k-2} \Big( \frac{\tau-1}{4}\Big)^{ k+\frac{1}{2} }  H_{2k+1} \Big( \frac{1+\tau}{2\sqrt{\tau-1}} \,t \Big) \frac{1}{t}. 
\end{split}
\end{align}
\end{prop}

\begin{proof}
We also mention that by \cite[Lemma 4.3]{BKLL23}, we have 
\begin{align}
	\begin{split} \label{RN1 integral form}
		R_{N}^{(1)}(x) & = R_{N}^{(1)}(0) - \sqrt{\frac{2}{\pi}} \frac{(\tau/2)^{N-\frac32}}{1+\tau} \frac{1}{(N-2)!}
 \int_{0}^{x} e^{ -\frac{ u^2 }{ 1+\tau }  } H_{N-2}\Big(  \frac{u}{ \sqrt{2\tau} } \Big) H_{N-1}\Big(  \frac{u}{ \sqrt{2\tau} } \Big)\,du
 \\
 &= \sqrt{\frac{2}{\pi}} \frac{(\tau/2)^{N-\frac32}}{1+\tau} \frac{1}{(N-2)!}
 \int_{x}^\infty e^{ -\frac{ u^2 }{ 1+\tau }  } H_{N-2}\Big(  \frac{u}{ \sqrt{2\tau} } \Big) H_{N-1}\Big(  \frac{u}{ \sqrt{2\tau} } \Big)\,du. 
	\end{split}
\end{align}
Using this and integration by parts, we obtain \eqref{u integral rep}. 
The second expression \eqref{u discrete rep} follows from \eqref{u integral rep} and the integral
\begin{align}
\begin{split}
&\quad \int_\R e^{-(x-y)^2} H_m(\alpha x) H_n(\alpha x) \,dx  = \sqrt{\pi} \sum_{k=0}^{ \min \{ m,n\} } 2^k \,k! \binom{m}{k} \binom{n}{k} (1-\alpha^2)^{ \frac{m+n}{2}-k }  H_{m+n-2k} \Big( \frac{\alpha y}{ \sqrt{1-\alpha^2}  } \Big)
\end{split}
\end{align}
that can be found in \cite[(7.374-9)]{GR14}. 
\end{proof}

We mention that the formula \eqref{RN1 integral form} plays an important role in the asymptotic analysis of the elliptic Ginibre ensembles, see \cite{LR16,BE22,BES23}. 
Indeed, the discrete representation \eqref{u discrete rep} will not play a role in the further analysis. 
Nonetheless, it can also provide an alternative proof for Proposition~\ref{Prop_u ODE} below. 
Furthermore, we stress that the expression \eqref{u discrete rep} is particularly useful for numerical verifications.

Next, we derive a differential equation for $u(t)$, which plays a key role in Theorem~\ref{Thm_M ODE}. 
For this, we shall make use of the integral representation \eqref{u integral rep}. 

\begin{prop}[\textbf{Differential equation for the MGF of the elliptic GinUE part}] \label{Prop_u ODE}
For any even integer $N \ge 2$ and $\tau \in [0,1]$, the function $u(t)$ satisfies the third-order differential equation:
\begin{equation} \label{u ODE}
a_3(t)\,u'''(t)+ a_2(t) \,u''(t)+a_1(t)\, u'(t)+ a_0(t) \,u(t)=0, 
\end{equation}
where 
\begin{align}
\begin{split}
a_3 (t)&= \frac{1-\tau}{2} ,
\qquad 
a_2(t)  = -\frac{ (1-4\tau+\tau^2)(1+\tau) }{ 4 }\,t+ \frac{1-\tau}{t},
\\
a_1(t) &= -(1+\tau) \bigg[ \frac{3\tau(1-\tau^2)}{8}\, t^2+ \frac{1-\tau^2}{2}(N-1) + \frac{1-5\tau+\tau^2}{2}     \bigg] - \frac{1-\tau}{t^2}  ,
\\
a_0(t) &= -\bigg[ \frac{\tau^2(1+\tau)}{8} \,t^2 + \frac{ \tau (1+\tau)}{2}  \,(N-1) + \frac{5\tau(1-\tau)}{8}     \bigg] (1+\tau)^2\,t. 
\end{split}
\end{align}
\end{prop}

\begin{rem}[\textbf{MGF of the GUE}]
Recall that the confluent hypergeometric functions $M(a,b,z)$ and $U(a,b,z)$ solve the Kummer's differential equation
\begin{equation} \label{Kummer ODE}
z \frac{ d^2w }{ dz^2 } +(b-z) \frac{dw}{dz} -a w=0,
\end{equation}
see e.g. \cite[Chapter 13]{NIST}.
Then it follows from \eqref{Hermite asymp} that if $\tau=1$, the expression \eqref{u discrete rep} reduces to 
\begin{align}
\begin{split} \label{u tau1}
 u(t) 
&  =   e^{ \frac{t^2}{2} }  \sum_{k=0}^{ N-2 }  \frac{1}{k!} \binom{N-1}{k+1}  \, t ^{ 2k } 
 = (N-1) e^{ \frac{t^2}{2} }  M(2-N,2,-t^2)  =\frac{ 1 }{(N-2)!}  e^{ \frac{t^2}{2} }  U(2-N,2,-t^2). 
\end{split}
\end{align}
We mention that the last expressions in \eqref{u tau1} were used in \cite{HT03}.
On the one hand, for $\tau=1$, the linear coefficients $a_k$ in Proposition~\ref{Prop_u ODE} are simplified as
\begin{equation}
a_3(t)=0, \qquad a_2(t)=t,\qquad a_1(t)=3,\qquad a_0(t)=-t\Big(t^2+4(N-1)\Big). 
\end{equation}
Then the resulting differential equation coincides with the differential equation derived in \cite[Eq. (18)]{Le09} with $N \mapsto N-1$, cf. \eqref{GOE density v1} and \eqref{GOE density v2}. 
One can also use \eqref{Kummer ODE} and \eqref{u tau1} to directly verify the differential equation \eqref{u ODE} for $\tau=1$.  
\end{rem}

We now formulate a differential equation that links the MGF $M(t)$ of the elliptic GinOE with the MGF $u(t)$ of the elliptic GinUE, restricted to the real axis.

\begin{prop}[\textbf{Differential equation for the mixed MGFs}] \label{Prop_ODE M u}
For any even integer $N \ge 2$ and $\tau \in [0,1]$, we have 
\begin{align}
\begin{split} \label{M u mixed ODE}
&\quad \frac{1-\tau}{2} M'''(t) -\frac{(1+\tau)(\tau^2-3\tau+1)}{2} \,t\,M''(t) 
 - \frac{1-\tau^2}{2} \Big( 2\tau(1+\tau) \, t^2 +(N-1)(1+\tau)+2-\tau \Big) M'(t)  
\\
&\quad -\frac{\tau(1+\tau)^2}{2} \Big( \tau(1+\tau) \,t^2+(N-1)(1+\tau)+3-2\tau \Big) \,t\,M(t)
\\
&=\frac{1-\tau}{2}u'''(t) -\frac{(1+\tau)(1-4\tau+\tau^2)}{4} \,t\, u''(t)
- (1+\tau)\Big( \frac{3\tau(1-\tau^2)}{8} \, t^2 +\frac{1-\tau^2}{2} (N-1)+\frac{1-\tau}{2} \Big) u'(t)
\\
&\quad  -\tau(1+\tau)^2 \Big( \frac{\tau(1+\tau)}{8} \,t^2+\frac{1+\tau}2 (N-1)+\frac{5+\tau}{8} \Big) \,t\,u(t). 
\end{split}
\end{align} 
\end{prop}

The differential equation in Proposition~\ref{Prop_ODE M u} can be used to derive a recurrence relation for the mixed moments of the elliptic GinOE and elliptic GinUE. 
However, it is a bit ambiguous, especially regarding the ``real moment'' of the elliptic GinUE. 
Nonetheless, in the Hermitian limit $\tau=1$, the statistical interpretation of the differential equation \eqref{M u mixed ODE} becomes clear, cf. \cite[Proposition 4]{Le09}. 
In particular, it gives rise to the recurrence relation 
\begin{align}
\begin{split} \label{M GOE GUE mixed}
 M_{2p}^{ \rm GOE } &= (4N-2) M_{2p-2}^{ \rm GOE } + 4(2p-2)(2p-3) M_{2p-4}^{ \rm GOE } 
 \\
 &\quad + M_{2p}^{ \rm GUE } -(4N-3) M_{2p-2}^{ \rm GUE } -(2p-2)(2p-3) M_{2p-4}^{ \rm GUE }
\end{split}
\end{align}
of the mixed spectral moments of the GOE and GUE established in \cite[Theorem 3]{Le09}. 

We now define
\begin{align}
\begin{split} \label{V definition}
V(t)&:= \frac{1-\tau}{2} M'''(t) -\frac{(1+\tau)(\tau^2-3\tau+1)}{2} \,t\,M''(t) 
\\
&\quad - \frac{1-\tau^2}{2} \Big( 2\tau(1+\tau) \, t^2 +(N-1)(1+\tau)+2-\tau \Big) M'(t)  
\\
&\quad -\frac{\tau(1+\tau)^2}{2} \Big( \tau(1+\tau) \,t^2+(N-1)(1+\tau)+3-2\tau \Big) \,t\,M(t).
\end{split}
\end{align}
Note that this is the left-hand side of the equation \eqref{M u mixed ODE}. 
In the following lemma, we relate the functions $u(t)$ and $V(t)$.

\begin{lem}[\textbf{Expression of $u$ in terms of $V$}] \label{Lem_u and V}
For any even integer $N \ge 2$ and $\tau \in [0,1]$, we have 
\begin{equation} \label{u abc V}
u(t)=\alpha(t) V(t)+\beta(t) V'(t) +\gamma(t) V''(t) , 
\end{equation}
where 
\begin{equation}
\alpha(t)= \frac{ a(t) }{ \delta(t) }, \qquad 
\beta(t)= \frac{ b(t) }{ \delta(t) }, \qquad 
\gamma(t)= \frac{ c(t) }{ \delta(t) }, \qquad 
\delta(t):= -N \tau^2 (1 + \tau)^5 \,d(t).
\end{equation}
Here $a(t),b(t),c(t)$ and $d(t)$ are given by \eqref{a(t) definition}, \eqref{b(t) definition}, \eqref{c(t) definition} and \eqref{d(t) definition}.
\end{lem}

We are now ready to prove Theorem~\ref{Thm_M ODE}.

\begin{proof}[Proof of Theorem~\ref{Thm_M ODE}]
It suffices to show that
\begin{equation} \label{V ODE}
B_4(t) V^{(4)}(t) + B_3(t) V'''(t)+B_2(t) V''(t)+B_1(t) V'(t)+B_0(t) V(t)=0,
\end{equation}
where $B_k$'s ($k=0,1,\dots,4$) are given in Subsection~\ref{Subsection beta jk}. 
Then the theorem follows by substituting \eqref{V definition} into \eqref{V ODE}.
By Proposition~\ref{Prop_ODE M u} and \eqref{V definition}, we have 
\begin{align}
\begin{split} \label{V in terms of u}
V(t)&=\frac{1-\tau}{2}u'''(t) -\frac{(1+\tau)(1-4\tau+\tau^2)}{4} \,t\, u''(t)
\\
&\quad - (1+\tau)\Big( \frac{3\tau(1-\tau^2)}{8} \, t^2 +\frac{1-\tau^2}{2} (N-1)+\frac{1-\tau}{2} \Big) u'(t)
\\
&\quad  -\tau(1+\tau)^2 \Big( \frac{\tau(1+\tau)}{8} \,t^2+\frac{1+\tau}2 (N-1)+\frac{5+\tau}{8} \Big) \,t\,u(t). 
\end{split}
\end{align}
Furthermore by using Proposition~\ref{Prop_u ODE}, we have
\begin{align}
\begin{split} \label{u ODE V}
 4(1-\tau)\,t\, u''(t) +\Big( 2\tau(1+\tau)(4-\tau) \,t^2 -4(1-\tau) \Big)  \, u'(t) + 3\tau^2  (1+\tau)^2 \,t^3\, u(t) + 4t^2\,V(t)=0. 
\end{split}
\end{align}
Then the desired formula \eqref{V ODE} follows from long but straightforward computations using Lemma~~\ref{Lem_u and V} and \eqref{u ODE V}. 
\end{proof}

By means of Theorem~\ref{Thm_M ODE}, we now complete the proof of Theorem~\ref{Thm_main}.

\begin{proof}[Proof of Theorem~\ref{Thm_main}]
Note that by \eqref{M Mp definition}, we have 
\begin{equation} \label{M tau derivatives}
M_\tau^{(k)}(t) = \sum_{j=k}^\infty \frac{ t^{j-k} }{(j-k)!} M_{j,\tau} = \sum_{ j=0 }^\infty  \frac{ t^j }{ j! }  M_{j+k,\tau}.
\end{equation}
We substitute this expression into \eqref{M ODE main} and collect the coefficient of $t^j$. This gives 
\begin{align*}
A_k(t) M_\tau^{(k)}(t) &=  \sum_{l=0}^{\lfloor (17-k)/2 \rfloor } \mathfrak{a}_{k,17-k-2l} \, t^{17-k-2l}  \, \sum_{ s=0 }^\infty  \frac{ M_{s+k,\tau} }{ s! }  \, t^s
= \sum_{j=0}^\infty   \bigg(  \sum_{l=0}^{\lfloor (17-k)/2 \rfloor }  \mathfrak{a}_{k,17-k-2l}  \frac{ M_{ j+2k+2l-17  } }{ (k+j+2l-17)!  }    \bigg)   \,t^j
\\
&= \sum_{j=0}^\infty   \bigg(  \sum_{l=0}^{\lfloor (17-k)/2 \rfloor }  \mathfrak{a}_{k,17-k-2l}  (j+k+2l-16)_{ 17-k-2l }\,M_{ j+2k+2l-17  }   \bigg)   \,\frac{t^j}{j!}.
\end{align*}
Then it follows from \eqref{M ODE main} that
\begin{equation}
\sum_{k=0}^7  \sum_{l=0}^{\lfloor (17-k)/2 \rfloor }  \mathfrak{a}_{k,17-k-2l}  (j+k+2l-16)_{ 17-k-2l }\,M_{ j+2k+2l-17  }     =0. 
\end{equation}
In particular, after some computations, it follows that the coefficient of $M_{j+3}$ is given by 
    \begin{align*}
    &\quad  \mathfrak{a}_{3,0}+ \mathfrak{a}_{4,1} j  +  \mathfrak{a}_{5,2} (j-1)_{2}  +  \mathfrak{a}_{6,3} (j-2)_{3} +  \mathfrak{a}_{7,4} (j-3)_{4}
    \\
    &= - 64 (j-5) (j-3) (j-1) (j+4) (1-\tau)^6
 \Big( 1  + 6 \tau + 2 N (1-\tau^2)\Big) \Big(4  + 17 \tau + 6 \tau^2 + 2 N(1- \tau^2) \Big)^2. 
    \end{align*}
By letting $j=2p-3,$ we conclude that for any integer $p \ge 10$, 
\begin{align}
\begin{split}
& \quad 2(2p+1)(1-\tau)^6
\Big( 1  + 6 \tau + 2 N (1-\tau^2)\Big) \Big(4  + 17 \tau + 6 \tau^2 + 2 N(1- \tau^2) \Big)^2 M_{2p}
\\
&=  \sum_{l=1}^{10} \bigg[ \sum_{ k=0 }^{ \min\{ 10-l, 7 \} } \frac{ (2p-k-2l+1)_{k+2l-3} }{ 
256(p-4) (p-3) (p-2) }\,\mathfrak{a}_{k,k+2l-3} \bigg] M_{2p-2l}.
\end{split}
\end{align} 
This completes the proof. 
\end{proof}


\section{Derivations of linear differential equations} \label{Section_ODEs}

In this section, we provide remaining proofs of Propositions~\ref{Prop_u ODE}, \ref{Prop_ODE M u} and Lemma~\ref{Lem_u and V}. 
Subsections~\ref{Subsection_Prop u ODE}, ~\ref{Subsection_Prop ODE M u} and ~\ref{Subsection_Prop uv in Lem} are devoted to the proofs of Propositions~\ref{Prop_u ODE}, \ref{Prop_ODE M u} and Lemma~\ref{Lem_u and V} respectively, the main steps for completing the proof of Theorem~\ref{Thm_M ODE}.

In the sequel, we shall frequently use the Gaussian integration by parts: for a differentiable function $f$ with polynomial growth, 
\begin{equation} \label{Gaussian IBP}
\int_\R x f(x) \, e^{ -\frac{ x^2 }{ 1+\tau }  } \,dx= \frac{1+\tau}{2} \int_\R  f'(x) \, e^{ -\frac{ x^2 }{ 1+\tau }  } \,dx. 
\end{equation}
We shall also often utilize the following version of the Gaussian integration by parts associated with the the Ornstein–Uhlenbeck operator: for differentiable functions $f$ and $g$ with polynomial growth,
\begin{equation} \label{Gaussian IBP 2}
\int_\R  f(x) \,\Big(  g''(x)-\frac{2}{1+\tau} \,x \,g'(x)  \Big) e^{ -\frac{ x^2 }{ 1+\tau }  } \,dx=  -\int_\R  f'(x) g'(x)  e^{ -\frac{ x^2 }{ 1+\tau }  } \,dx.
\end{equation}

\subsection{Proof of Proposition~\ref{Prop_u ODE}} \label{Subsection_Prop u ODE}

In this subsection, we prove Proposition~\ref{Prop_u ODE}. 
For this, we shall utilize the integral representation \eqref{u integral rep} of $u(t)$. 
We first derive a differential equation for a related but simpler function $\sigma$ to analyse, for which we replace the integrand $H_{N-2}H_{N-1}$ in \eqref{u integral rep} with $H_{N-1}^2$.

\begin{lem} \label{Lem_ODE for sigma}
For even integer $N \ge 2$ and $\tau \in [0,1]$, let
\begin{equation} \label{sigma definition}
\sigma(t):=  \frac{2 }{ 1+\tau }   \frac{ (\tau/2)^{N-\frac32 } }{(N-2)!}  \int_\R  e^{tx}\,   H_{N-1}\Big(\frac{x}{ \sqrt{2\tau} }\Big)^2  \frac{ e^{ -\frac{ x^2 }{ 1+\tau }  }  }{ \sqrt{2\pi} }  \, dx.
\end{equation}
Then we have 
\begin{align}
\begin{split}
 \frac{4(1-\tau)}{1+\tau} \,\sigma'''(t)  & =  2(1-4\tau+\tau^2) \,t \,\sigma''(t)
+ \bigg[ 3\tau(1-\tau^2)\, t^2 + 4 \Big( (N-1)(1-\tau^2)+1-2\tau\Big)  \bigg] \sigma'(t)
\\
&\quad   + \bigg[  \tau^2(1+\tau)\, t^2+  4(N-1)\tau(1+\tau)+\tau(5-\tau ) \bigg] (1+\tau) \,t\, \sigma(t).
\end{split}
\end{align}
\end{lem}

\begin{proof}
We first write  
$$
\sigma(t)=  c  \int_\R  e^{tx -\frac{ x^2 }{ 1+\tau } }\,   \phi(x) \, dx, \qquad 
\phi(x):= H_{N-1}\Big(\frac{x}{ \sqrt{2\tau} }\Big)^2, \qquad c= \frac{1  }{ \sqrt{2\pi} }  \frac{2 }{ 1+\tau }   \frac{ (\tau/2)^{N-\frac32 } }{(N-2)!}  .
$$
Note that by the well-known differentiation rule 
\begin{equation}  \label{Hermite diff rule}
H_k'(x)= 2k \,H_{k-1}(x)
\end{equation}
of the Hermite polynomials, after some manipulations, we have 
\begin{equation}
\Big[ 4(N-1)x-( 4 (N-1) \tau-\tau +2x^2 ) \partial_x +3x\tau \partial_x^2 -\tau^2 \partial_x^3 \Big] \phi(x) =0. 
\end{equation}
Using integration by parts, we have 
$$
\int_\R e^{ tx -\frac{x^2}{1+\tau} } \partial_x \phi(x)\,dx = \int_\R  \phi(x) \Big( \frac{2x}{1+\tau} -t \Big) e^{ tx -\frac{x^2}{1+\tau} } \,dx  = \int_\R  \phi(x) \Big( \frac{2 \partial_t }{1+\tau} -t \Big) e^{ tx -\frac{x^2}{1+\tau} } \,dx.
$$
Then the desired equation follows from 
$$
\int_\R e^{ tx -\frac{x^2}{1+\tau} } \,x\,\phi(x)\,dx = \partial_t \int_\R e^{ tx -\frac{x^2}{1+\tau} } \phi(x)\,dx .
$$
\end{proof}

We are now ready to prove Proposition~\ref{Prop_u ODE}

\begin{proof}[Proof of Proposition~\ref{Prop_u ODE}]
Let us define
\begin{equation}  \label{rho definition}
\rho(t):= \frac{2 }{ 1+\tau }   \frac{ (\tau/2)^{N-\frac32 } }{(N-2)!}  \int_\R  e^{tx}\,  H_{N-2}\Big(\frac{x}{\sqrt{2\tau}} \Big) H_{N-1}\Big(\frac{x}{ \sqrt{2\tau} }\Big)  \frac{ e^{ -\frac{ x^2 }{ 1+\tau }  }  }{ \sqrt{2\pi} }  \, dx. 
\end{equation}
Notice that by definition, we have 
\begin{equation} \label{u rho relation}
t\, u(t) = \rho(t).
\end{equation}    
We also write 
\begin{equation}
 \widetilde{\rho}(t) := (N-1) \sqrt{ \frac{2}{\tau} } (1+\tau)\rho(t).
\end{equation}
Note that  
\begin{align}
\begin{split} \label{sigma' sigma}
\sigma'(t) - \frac{1+\tau}{2} \,t\,\sigma(t) =    \frac{ (\tau/2)^{N-2 } }{(N-2)!}  \int_\R   e^{tx}\,    H_{N-1}\Big(\frac{x}{ \sqrt{2\tau} }\Big) H_{N-1}'\Big(\frac{x}{ \sqrt{2\tau} }\Big)   \frac{ e^{ -\frac{ x^2 }{ 1+\tau }  }  }{ \sqrt{2\pi} }  \, dx.
\end{split}
\end{align} 
Then by \eqref{sigma' sigma} and \eqref{Hermite diff rule}, we have 
\begin{equation}
 \label{rho tilde sigma}
\widetilde{\rho}(t) = 2\sigma'(t) - (1+\tau) \,t\,\sigma(t) .
\end{equation}
By taking derivatives, we have 
\begin{align}
\begin{split}
&\quad -\frac{2(1-\tau)}{(1+\tau)} \widetilde{\rho}''(t) -\tau(3-\tau) \,t \, \widetilde{\rho}'(t) -\Big( \tau^2(1+\tau)t^2-2(N-1)(1-\tau^2)+2\tau \Big)  \widetilde{\rho}(t)
\\
&=  -\frac{4(1-\tau)}{(1+\tau)}  \, \sigma'''(t) +2 (1-4\tau+\tau^2) \,t\, \sigma''(t)  + \bigg[ 3\tau(1-\tau^2)\, t^2 + 4 \Big( (N-1)(1-\tau^2)+1-2\tau\Big)  \bigg] \sigma'(t)
\\
& \quad + \bigg[   \tau^2(1+\tau)t^2-2(N-1)(1-\tau^2) +\tau(5-\tau)  \bigg]  (1+\tau)\,t \, \sigma(t). 
\end{split}
\end{align}
Combining this with Lemma~\ref{Lem_ODE for sigma}, we obtain  
\begin{align}
\begin{split} \label{sigma rho 1}
 2(N-1)(1+\tau)^3 \,t\,\sigma(t)   
& =  \frac{2(1-\tau)}{(1+\tau)} \widetilde{\rho}''(t)+ \tau(3-\tau) \,t \, \widetilde{\rho}'(t) 
+\Big( \tau^2(1+\tau) \, t^2-2(N-1)(1-\tau^2)+2\tau \Big)  \widetilde{\rho}(t). 
\end{split}
\end{align}
Differentiating this expression, we obtain 
\begin{align}
\begin{split} \label{sigma rho 2}
&\quad  2(N-1)(1+\tau)^3 \,\Big( \sigma(t)+t\,\sigma'(t)\Big) 
=  \frac{2(1-\tau)}{(1+\tau)} \widetilde{\rho}'''(t)+ \tau(3-\tau) \,t \, \widetilde{\rho}''(t) 
\\
&\quad +\Big( \tau^2(1+\tau) \, t^2-2(N-1)(1-\tau^2)+\tau(5-\tau) \Big)  \widetilde{\rho}'(t)+ 2\tau^2(1+\tau)\,t\, \widetilde{\rho}(t).  
\end{split}
\end{align}
Then by the linear combination of \eqref{sigma rho 1} and \eqref{sigma rho 2} together with the definition of $\widetilde{\rho}$ in \eqref{rho tilde sigma}, we conclude
\begin{align}
\begin{split}
0 &= \frac{4(1-\tau)}{(1+\tau)}\,t\, \rho'''(t) -\Big( 2(1-4\tau+\tau^2) \, t^2   +\frac{4(1-\tau)}{(1+\tau)}  \Big) \, \rho''(t)
\\
&\quad - \bigg[ 3\tau(1-\tau^2) \, t^2+4(N-1)(1-\tau^2)- 4\tau   \bigg] t\, \rho'(t)
\\
&\quad  -\bigg[ \tau^2(1+\tau)^2 \,t^4 + 2\tau(1-\tau^2)\,t^2 + 4\tau +4(1+\tau) \Big( \tau(1+\tau)\,t^2-(1-\tau)\Big)(N-1)  \bigg] \rho(t).
\end{split}
\end{align}
Now the proposition follows from the relation \eqref{u rho relation}. 
\end{proof}

\subsection{Proof of Proposition~\ref{Prop_ODE M u}} \label{Subsection_Prop ODE M u}

In this subsection, we prove Proposition~\ref{Prop_ODE M u}. 
Recall that by \eqref{M u v}, we have $M(t)=u(t)+v(t)$.
Therefore, it suffices to show the following equivalent proposition. 

\begin{prop} \label{Prop_v u ODE}
For any even integer $N \ge 2$ and $\tau \in [0,1]$, we have  
\begin{align}
\begin{split} \label{v u ODE}
&\quad 4(1-\tau) v'''(t) -4(1+\tau)(\tau^2-3\tau+1) \,t\,v''(t) 
\\
&\quad - 4(1-\tau^2)\Big( 2\tau(1+\tau) \, t^2 +(N-1)(1+\tau)+2-\tau \Big) v'(t)  
\\
&\quad -4\tau(1+\tau)^2 \Big( \tau(1+\tau) \,t^2+(N-1)(1+\tau)+3-2\tau \Big) \,t\,v(t)
\\
&= 2(1-\tau)^2(1+\tau) \,t\, u''(t)+(1-\tau^2) \Big( 5\tau(1+\tau)\,t^2+4(1-\tau)\Big) u'(t) 
\\
&\quad +\tau(1+\tau)^2 \Big( 3\tau(1+\tau) \,t^2+7-9\tau \Big) \,t \,u(t). 
\end{split}
\end{align}
\end{prop}

For this, we first prove the following lemma. 

\begin{lem} \label{Lem_v rho v0}
For any even integer $N \ge 2$ and $\tau \in [0,1]$, we have  
\begin{align}
\begin{split} \label{v rho H int}
&\quad  \frac{2}{1+\tau} v''(t) -2(1-\tau) \,t\,v'(t)-2 \Big( \tau(1+\tau) \,t^2+ (1+\tau) (N-1)+1 \Big) v(t) 
\\
&= \rho'(t)+  \tau(1+\tau) \,t\,\rho(t) - \frac{ (\tau/2)^{N-\frac32} }{ (N-2)! }   \sqrt{ 2 \tau}  \int_\R  e^{t x } \, \frac{ e^{ -\frac{x^2 }{1+\tau} } }{\sqrt{2\pi}}   H_{N-2} \Big( \frac{x}{ \sqrt{2\tau} } \Big)  \, H_{N-1}'\Big( \frac{x}{\sqrt{2\tau}} \Big) \,dx, 
\end{split}
\end{align}
where $\rho$ is given by \eqref{rho definition}. 
\end{lem}
\begin{proof}
Note that by \eqref{RN2 def}, we have
\begin{equation}
v(t)= \frac{(\tau/2)^{N-\frac32} }{ 1+\tau } \frac{ 1 }{ (N-2)! }    \int_\R e^{t x } \, \frac{ e^{ -\frac{x^2 }{2(1+\tau)} } }{\sqrt{2\pi}}  H_{N-1}\Big( \frac{x}{ \sqrt{2\tau} } \Big)\,  
	\bigg[ \int_{0}^{x} e^{ -\frac{ u^2 }{2(1+\tau)}  } H_{N-2} \Big( \frac{u}{ \sqrt{2\tau} } \Big)\,du \bigg] \,dx. 
\end{equation}
Therefore we have 
\begin{align*}
v'(t)  = \frac{ (\tau/2)^{N-\frac32}  }{ (N-2)! }    \int_\R  \frac{d}{dx} \bigg[  e^{t x } \, \frac{ e^{ -\frac{x^2 }{2(1+\tau)} } }{\sqrt{2\pi}}  H_{N-1}\Big( \frac{x}{ \sqrt{2\tau} } \Big)\,  
	 \int_{0}^{x} e^{ -\frac{ u^2 }{2(1+\tau)}  } H_{N-2} \Big( \frac{u}{ \sqrt{2\tau} } \Big)\,du \bigg] \,dx. 
\end{align*}
Here, by using the definition of $\rho$ in \eqref{rho definition}, we have
\begin{align}
\begin{split}  \label{v' v rho}
v'(t)&= (1+\tau)\,t\, v(t)   + \frac{1+\tau}{2} \rho(t) 
\\
&\quad + \frac{1}{\sqrt{2\tau}} \frac{ (\tau/2)^{N-\frac32}  }{ (N-2)! }  \int_\R   e^{tx}\, H_{N-1}'\Big( \frac{x}{ \sqrt{2\tau} } \Big) e^{ -\frac{ x^2 }{2(1+\tau)}  }  \bigg[ \int_{0}^{x} H_{N-2} \Big( \frac{u}{ \sqrt{2\tau} } \Big) \,\frac{ e^{ -\frac{ u^2 }{2(1+\tau)}  } }{ \sqrt{2\pi}  } \, du  \bigg]\,dx. 
\end{split}
\end{align}
By differentiating this, we also obtain 
\begin{align}
\begin{split}  \label{v'' v' v rho}
v''(t)&= (1+\tau)\, v(t)+(1+\tau)\,t\,v'(t)   + \frac{1+\tau}{2} \rho'(t) 
\\
&\quad + \frac{1}{\sqrt{2\tau}} \frac{ (\tau/2)^{N-\frac32}  }{ (N-2)! }  \int_\R  x\, e^{tx}\, H_{N-1}'\Big( \frac{x}{ \sqrt{2\tau} } \Big) e^{ -\frac{ x^2 }{2(1+\tau)}  }  \bigg[ \int_{0}^{x} H_{N-2} \Big( \frac{u}{ \sqrt{2\tau} } \Big) \,\frac{ e^{ -\frac{ u^2 }{2(1+\tau)}  } }{ \sqrt{2\pi}  } \, du \bigg]\,dx .
\end{split}
\end{align}

We now make use of the recurrence relation 
\begin{equation} \label{Hk Hk'' Hk'}
H_k''(x)-2 x H_k'(x) = -2 k H_k(x). 
\end{equation} to write  
\begin{align}
\begin{split} \label{v N-1}
&\quad (N-1) v(t) = -\frac12 \frac{(\tau/2)^{N-\frac32} }{ 1+\tau } \frac{ 1 }{ (N-2)! }  
\\
&\quad \times \int_\R e^{t x } \, \frac{ e^{ -\frac{x^2 }{2(1+\tau)} } }{\sqrt{2\pi}}  \bigg[  H_{N-1}''\Big( \frac{x}{ \sqrt{2\tau} }\Big)- \sqrt{ \frac{2}{\tau} } \,x\, H_{N-1}'\Big( \frac{x}{ \sqrt{2\tau} }\Big) \bigg] 
\bigg[ \int_{0}^{x} e^{ -\frac{ u^2 }{2(1+\tau)}  } H_{N-2} \Big( \frac{u}{ \sqrt{2\tau} } \Big)\,du \bigg] \,dx. 
\end{split}
\end{align}
As before, we write 
\begin{align*}
&\quad \int_\R e^{t x } \, \frac{ e^{ -\frac{x^2 }{2(1+\tau)} } }{\sqrt{2\pi}}  \bigg[  H_{N-1}''\Big( \frac{x}{ \sqrt{2\tau} }\Big)- \sqrt{ \frac{2}{\tau} } \,x\, H_{N-1}'\Big( \frac{x}{ \sqrt{2\tau} }\Big) \bigg] 
\bigg[ \int_{0}^{x} e^{ -\frac{ u^2 }{2(1+\tau)}  } H_{N-2} \Big( \frac{u}{ \sqrt{2\tau} } \Big)\,du \bigg] \,dx
\\
&= \int_\R e^{t x } \, \frac{ e^{ -\frac{x^2 }{2(1+\tau)} } }{\sqrt{2\pi}}   \bigg[  H_{N-1}''\Big( \frac{x}{ \sqrt{2\tau} }\Big)- \frac{ \sqrt{2\tau} }{ 1+\tau }  \,x\, H_{N-1}'\Big( \frac{x}{ \sqrt{2\tau} }\Big) \bigg] 
\bigg[ \int_{0}^{x} e^{ -\frac{ u^2 }{2(1+\tau)}  } H_{N-2} \Big( \frac{u}{ \sqrt{2\tau} } \Big)\,du \bigg] \,dx
\\
&\quad -  \frac{ 1 }{ 1+\tau }  \sqrt{ \frac{2}{\tau} } \int_\R e^{t x } \, \frac{ e^{ -\frac{x^2 }{2(1+\tau)} } }{\sqrt{2\pi}}   \,x\, H_{N-1}'\Big( \frac{x}{ \sqrt{2\tau} }\Big) 
\bigg[ \int_{0}^{x} e^{ -\frac{ u^2 }{2(1+\tau)}  } H_{N-2} \Big( \frac{u}{ \sqrt{2\tau} } \Big)\,du \bigg] \,dx.
\end{align*}
Then by applying \eqref{Gaussian IBP 2} with 
$$
f(x)= e^{t x } \,   \int_{0}^{x} \frac{ e^{ -\frac{ u^2 }{2(1+\tau)}  } }{\sqrt{2\pi}}  H_{N-2} \Big( \frac{u}{ \sqrt{2\tau} } \Big)\,du , \qquad g(x) = 2 \tau \, H_{N-1}\Big( \frac{x}{\sqrt{2\tau}} \Big),    
$$
it follows that 
\begin{align*}
&\quad \int_\R e^{t x } \, \frac{ e^{ -\frac{x^2 }{2(1+\tau)} } }{\sqrt{2\pi}}  \bigg[  H_{N-1}''\Big( \frac{x}{ \sqrt{2\tau} }\Big)- \sqrt{ \frac{2}{\tau} } \,x\, H_{N-1}'\Big( \frac{x}{ \sqrt{2\tau} }\Big) \bigg] 
\bigg[ \int_{0}^{x} e^{ -\frac{ u^2 }{2(1+\tau)}  } H_{N-2} \Big( \frac{u}{ \sqrt{2\tau} } \Big)\,du \bigg] \,dx
\\
&= - \sqrt{ 2 \tau}  \int_\R  e^{t x } \, \frac{ e^{ -\frac{x^2 }{1+\tau} } }{\sqrt{2\pi}}   H_{N-2} \Big( \frac{x}{ \sqrt{2\tau} } \Big)  \, H_{N-1}'\Big( \frac{x}{\sqrt{2\tau}} \Big) \,dx
\\
&\quad -\sqrt{2\tau} \,t \int_\R  e^{t x } \, \frac{ e^{ -\frac{x^2 }{2(1+\tau)} } }{\sqrt{2\pi}} \,H_{N-1}'\Big( \frac{x}{\sqrt{2\tau}} \Big) \bigg[ \int_{0}^{x} e^{ -\frac{ u^2 }{2(1+\tau)}  } H_{N-2} \Big( \frac{u}{ \sqrt{2\tau} } \Big)\,du\bigg] \,dx
\\
&\quad -  \frac{ 1 }{ 1+\tau }  \sqrt{ \frac{2}{\tau} } \int_\R e^{t x } \, \frac{ e^{ -\frac{x^2 }{2(1+\tau)} } }{\sqrt{2\pi}}   \,x\, H_{N-1}'\Big( \frac{x}{ \sqrt{2\tau} }\Big) 
\bigg[ \int_{0}^{x} e^{ -\frac{ u^2 }{2(1+\tau)}  } H_{N-2} \Big( \frac{u}{ \sqrt{2\tau} } \Big)\,du \bigg] \,dx.
\end{align*}
Furthermore, by using \eqref{v' v rho} and \eqref{v N-1}, we obtain
\begin{align}
\begin{split}
&\quad \Big( 2\tau(1+\tau) \,t^2+ 2(1+\tau) (N-1) \Big) v(t) - 2\tau  \,t\, v'(t)+  \tau(1+\tau) \,t\,\rho(t)
\\
&=  \frac{ (\tau/2)^{N-\frac32} }{ (N-2)! }   \sqrt{ 2 \tau}  \int_\R  e^{t x } \, \frac{ e^{ -\frac{x^2 }{1+\tau} } }{\sqrt{2\pi}}   H_{N-2} \Big( \frac{x}{ \sqrt{2\tau} } \Big)  \, H_{N-1}'\Big( \frac{x}{\sqrt{2\tau}} \Big) \,dx
\\
&\quad + \frac{ (\tau/2)^{N-\frac32} }{ (N-2)! }   \frac{ 1 }{ 1+\tau }  \sqrt{ \frac{2}{\tau} } \int_\R e^{t x } \, \frac{ e^{ -\frac{x^2 }{2(1+\tau)} } }{\sqrt{2\pi}}   \,x\, H_{N-1}'\Big( \frac{x}{ \sqrt{2\tau} }\Big) 
\bigg[ \int_{0}^{x} e^{ -\frac{ u^2 }{2(1+\tau)}  } H_{N-2} \Big( \frac{u}{ \sqrt{2\tau} } \Big)\,du \bigg] \,dx.
\end{split}
\end{align}
Combining this with \eqref{v'' v' v rho}, the lemma follows. 
\end{proof}

We now complete the proof of Proposition~\ref{Prop_v u ODE}.

\begin{proof}[Proof of Proposition~\ref{Prop_v u ODE}]

By taking the derivatives of \eqref{rho definition} and using the Gaussian integration by parts \eqref{Gaussian IBP}, we have 
\begin{align} \label{rho' rho}
\begin{split}
&\quad \rho'(t)- \frac{1+\tau}{2} \,t\,\rho(t) 
\\
&=  \frac{ (\tau/2)^{N-\frac32 } }{(N-2)!} \frac{1}{ \sqrt{2\tau} } \int_\R  e^{tx}\,\bigg[ H_{N-2}'\Big(\frac{x}{\sqrt{2\tau}} \Big) H_{N-1}\Big(\frac{x}{ \sqrt{2\tau} }\Big) +  H_{N-2}\Big(\frac{x}{\sqrt{2\tau}} \Big) H_{N-1}'\Big(\frac{x}{ \sqrt{2\tau} }\Big) \bigg] \frac{ e^{ -\frac{ x^2 }{ 1+\tau }  }  }{ \sqrt{2\pi} }  \, dx 
\end{split}
\end{align}
and 
\begin{align}
\begin{split} \label{rho '' rho' rho}
&\quad 2\tau \rho''(t)- \tau(1+\tau) \,t\,\rho'(t) -\tau(1+\tau) \rho(t)
\\
&=  \frac{ 2(\tau/2)^{N-1 } }{(N-2)!}   \int_\R  x\,e^{tx}\,\bigg[ H_{N-2}'\Big(\frac{x}{\sqrt{2\tau}} \Big) H_{N-1}\Big(\frac{x}{ \sqrt{2\tau} }\Big) +  H_{N-2}\Big(\frac{x}{\sqrt{2\tau}} \Big) H_{N-1}'\Big(\frac{x}{ \sqrt{2\tau} }\Big) \bigg] \frac{ e^{ -\frac{ x^2 }{ 1+\tau }  }  }{ \sqrt{2\pi} }  \, dx .
\end{split}
\end{align}
Furthermore, by using \eqref{Hk Hk'' Hk'} and a similar decomposition as above, we have
\begin{align*}
(N-1)\rho(t)&=  \frac{ 1 }{ 1+\tau }   \frac{ (\tau/2)^{N-\frac32 } }{(N-2)!}  \sqrt{2\tau}\, t \, \int_\R   e^{tx}\,   H_{N-2}\Big(\frac{x}{ \sqrt{2\tau} }\Big)   H_{N-1}'\Big(\frac{x}{ \sqrt{2\tau} }\Big)   \frac{ e^{ -\frac{ x^2 }{ 1+\tau }  }  }{ \sqrt{2\pi} }  \, dx 
\\
&\quad + \frac{ 1 }{ 1+\tau }   \frac{ (\tau/2)^{N-\frac32 } }{(N-2)!}  \int_\R  e^{tx}  H_{N-2}'\Big(\frac{x}{ \sqrt{2\tau} }\Big) H_{N-1}'\Big(\frac{x}{ \sqrt{2\tau} }\Big) \frac{ e^{ -\frac{ x^2 }{ 1+\tau }  }  }{ \sqrt{2\pi} }  \, dx
\\
&\quad - \frac{ 1 }{ 1+\tau }   \frac{ (\tau/2)^{N-\frac32 } }{(N-2)!} 
 \Big( \frac{ 2\sqrt{2\tau} }{ 1+\tau } - \sqrt{ \frac{2}{\tau} } \Big) \int_\R  e^{tx}\,   H_{N-2}\Big(\frac{x}{ \sqrt{2\tau} }\Big)  \,x\,  H_{N-1}'\Big( \frac{x}{ \sqrt{2\tau} }\Big)    \frac{ e^{ -\frac{ x^2 }{ 1+\tau }  }  }{ \sqrt{2\pi} }  \, dx 
\end{align*}
and 
\begin{align*}
(N-2)\rho(t)&=  \frac{ 1 }{ 1+\tau }   \frac{ (\tau/2)^{N-\frac32 } }{(N-2)!}  \sqrt{2\tau}\, t \, \int_\R   e^{tx}\,   H_{N-1}\Big(\frac{x}{ \sqrt{2\tau} }\Big)   H_{N-2}'\Big(\frac{x}{ \sqrt{2\tau} }\Big)   \frac{ e^{ -\frac{ x^2 }{ 1+\tau }  }  }{ \sqrt{2\pi} }  \, dx 
\\
&\quad + \frac{ 1 }{ 1+\tau }   \frac{ (\tau/2)^{N-\frac32 } }{(N-2)!}  \int_\R  e^{tx} H_{N-1}'\Big(\frac{x}{ \sqrt{2\tau} }\Big) H_{N-2}'\Big(\frac{x}{ \sqrt{2\tau} }\Big) \frac{ e^{ -\frac{ x^2 }{ 1+\tau }  }  }{ \sqrt{2\pi} }  \, dx
\\
&\quad - \frac{ 1 }{ 1+\tau }   \frac{ (\tau/2)^{N-\frac32 } }{(N-2)!} 
 \Big( \frac{ 2\sqrt{2\tau} }{ 1+\tau } - \sqrt{ \frac{2}{\tau} } \Big) \int_\R  e^{tx}\,   H_{N-1}\Big(\frac{x}{ \sqrt{2\tau} }\Big)  \,x\,  H_{N-2}'\Big( \frac{x}{ \sqrt{2\tau} }\Big)    \frac{ e^{ -\frac{ x^2 }{ 1+\tau }  }  }{ \sqrt{2\pi} }  \, dx .
\end{align*}
Subtracting these equations, we obtain
\begin{align*}
\begin{split}
\rho(t) & =  \frac{ \sqrt{2\tau} }{ 1+\tau }   \frac{ (\tau/2)^{N-\frac32 } }{(N-2)!}  \, t \, \int_\R   e^{tx}\,  \bigg[ H_{N-2}\Big(\frac{x}{ \sqrt{2\tau} }\Big)   H_{N-1}'\Big(\frac{x}{ \sqrt{2\tau} }\Big)-H_{N-2}'\Big(\frac{x}{ \sqrt{2\tau} }\Big)   H_{N-1}\Big(\frac{x}{ \sqrt{2\tau} }\Big)   \bigg]  \frac{ e^{ -\frac{ x^2 }{ 1+\tau }  }  }{ \sqrt{2\pi} }  \, dx 
\\
&\quad +  \frac{ (\tau/2)^{N-2 } }{(N-2)!}  \frac{1-\tau}{ (1+\tau)^2 }  \int_\R x\, e^{tx}\, \bigg[  H_{N-2}\Big(\frac{x}{ \sqrt{2\tau} }\Big)  H_{N-1}'\Big( \frac{x}{ \sqrt{2\tau} }\Big)-H_{N-2}'\Big(\frac{x}{ \sqrt{2\tau} }\Big)  H_{N-1}\Big( \frac{x}{ \sqrt{2\tau} }\Big)   \bigg]  \frac{ e^{ -\frac{ x^2 }{ 1+\tau }  }  }{ \sqrt{2\pi} }  \, dx. 
\end{split}
\end{align*}
Combining this expression with \eqref{rho' rho}, we obtain
\begin{align*}
\begin{split} 
 &\quad \tau \,\rho'(t)- \frac{\tau(1+\tau)}{2} \,t\,\rho(t)  + \frac{1+\tau}{2}\frac{ \rho(t) }{ t }
 \\
 &=   \frac{ (\tau/2)^{N-\frac32 } }{(N-2)!} \sqrt{2\tau} \, \int_\R   e^{tx}\, H_{N-2}\Big(\frac{x}{ \sqrt{2\tau} }\Big)   H_{N-1}'\Big(\frac{x}{ \sqrt{2\tau} }\Big)   \frac{ e^{ -\frac{ x^2 }{ 1+\tau }  }  }{ \sqrt{2\pi} }  \, dx 
 \\
&\quad +  \frac{ (\tau/2)^{N-\frac32 } }{(N-2)!} \frac{1}{ \sqrt{2\tau} } \frac{1-\tau}{ 1+\tau } \, \frac{1}{t}
 \int_\R x\, e^{tx}\, \bigg[  H_{N-2}\Big(\frac{x}{ \sqrt{2\tau} }\Big)  H_{N-1}'\Big( \frac{x}{ \sqrt{2\tau} }\Big)-H_{N-2}'\Big(\frac{x}{ \sqrt{2\tau} }\Big)  H_{N-1}\Big( \frac{x}{ \sqrt{2\tau} }\Big)   \bigg]  \frac{ e^{ -\frac{ x^2 }{ 1+\tau }  }  }{ \sqrt{2\pi} }  \, dx .
\end{split}
\end{align*}
Then it follows from Lemma~\ref{Lem_v rho v0} that
\begin{align*}
\begin{split}
&\quad  \frac{2}{1+\tau} v''(t) -2(1-\tau) \,t\,v'(t)-2 \Big( \tau(1+\tau) \,t^2+ (1+\tau) (N-1)+1 \Big) v(t) 
\\
&= (1-\tau) \rho'(t)+ \frac{1+\tau}{2} \Big( 3 
\tau \,t  - \frac{ 1 }{ t } \Big) \rho(t)
\\
&\quad +  \frac{ (\tau/2)^{N-\frac32 } }{(N-2)!} \frac{1}{ \sqrt{2\tau} } \frac{1-\tau}{ 1+\tau } \, \frac{1}{t}  \int_\R x\, e^{tx}\, \bigg[  H_{N-2}\Big(\frac{x}{ \sqrt{2\tau} }\Big)  H_{N-1}'\Big( \frac{x}{ \sqrt{2\tau} }\Big)-H_{N-2}'\Big(\frac{x}{ \sqrt{2\tau} }\Big)  H_{N-1}\Big( \frac{x}{ \sqrt{2\tau} }\Big)   \bigg]  \frac{ e^{ -\frac{ x^2 }{ 1+\tau }  }  }{ \sqrt{2\pi} }  \, dx .
\end{split}
\end{align*}

On the other hand, by differentiating \eqref{v rho H int} in Lemma~\ref{Lem_v rho v0} and using \eqref{rho '' rho' rho}, we have 
\begin{align*}
&\quad - \frac{ (\tau/2)^{N-\frac32 } }{(N-2)!} \sqrt{2\tau}  \int_\R  x\,e^{tx}\,\bigg[  H_{N-2}'\Big(\frac{x}{\sqrt{2\tau}} \Big) H_{N-1}\Big(\frac{x}{ \sqrt{2\tau} }\Big) -  H_{N-2}\Big(\frac{x}{\sqrt{2\tau}} \Big) H_{N-1}'\Big(\frac{x}{ \sqrt{2\tau} }\Big) \bigg] \frac{ e^{ -\frac{ x^2 }{ 1+\tau }  }  }{ \sqrt{2\pi} }  \, dx
\\
&=  2(1-\tau) \rho''(t)+ 3\tau(1+\tau) \,t\,\rho'(t) + 3\tau(1+\tau) \rho(t)  
\\
&\quad  -\frac{4}{1+\tau} v'''(t)  + 4(1-\tau) \,t\,v''(t)+  4 \Big( \tau(1+\tau) \,t^2+ (1+\tau) (N-1)+2-\tau \Big) v'(t)  + 8  \tau(1+\tau) \,t\, v(t)  . 
\end{align*}
Substituting this into the above equation, after simplifications, we obtain 
\begin{align}
\begin{split}
&\quad 4(1-\tau) v'''(t) -4(1+\tau)(\tau^2-3\tau+1) \,t\,v''(t) 
\\
&\quad - 4(1-\tau^2)\Big( 2\tau(1+\tau) \, t^2 +(N-1)(1+\tau)+2-\tau \Big) v'(t)  
\\
&\quad -4\tau(1+\tau)^2 \Big( \tau(1+\tau) \,t^2+(N-1)(1+\tau)+3-2\tau \Big) \,t\,v(t)
\\
&= 2(1-\tau)^2(1+\tau) \rho''(t) +5\tau(1-\tau)(1+\tau)^2  \,t\,\rho'(t)
 +\tau(1+\tau)^2\Big( 3\tau(1+\tau) \,t^2+2-4\tau \Big) \rho(t). 
\end{split}
\end{align}
Finally, the desired equation \eqref{v u ODE} follows from \eqref{u rho relation}.
\end{proof}

\subsection{Proof of Lemma~\ref{Lem_u and V}}  \label{Subsection_Prop uv in Lem}

Next, we prove Lemma~\ref{Lem_u and V}, which relies on straightforward computations using Proposition~\ref{Prop_u ODE}.

\begin{proof}[Proof of Lemma~\ref{Lem_u and V}]
By multiplying $t$ in the expression \eqref{V in terms of u} and taking  the derivative, we have 
\begin{align*}
V(t)+t\,V'(t) &= -(1-\tau) u'''(t) - \Big( \frac{\tau(1+\tau)(4-\tau)}{2} \,t -\frac{1-\tau}{t} \Big)  \, u''(t) 
\\
&\quad - \Big( \frac{3\tau^2  (1+\tau)^2}{4} \,t^2+ \frac{\tau(1+\tau)(4-\tau)}{2}  +\frac{1-\tau}{t^2} \Big)  \, u'(t) - \frac{3\tau^2  (1+\tau)^2}{2} \,t\, u'(t).
\end{align*}
Then by using Proposition~\ref{Prop_u ODE}, 
\begin{align*}
&\quad V(t)+t\,V'(t) = \Big( 2a_3(t) -(1-\tau) \Big) u'''(t) + \bigg[ 2a_2(t) - \Big( \frac{\tau(1+\tau)(4-\tau)}{2} \,t -\frac{1-\tau}{t} \Big) \bigg] \, u''(t) 
\\
&\quad +\bigg[2a_1(t)- \Big( \frac{3\tau^2  (1+\tau)^2}{4} \,t^2+ \frac{\tau(1+\tau)(4-\tau)}{2}  +\frac{1-\tau}{t^2} \Big)  \bigg]\, u'(t) + \bigg[2a_0(t)- \frac{3\tau^2  (1+\tau)^2}{2}\,t \bigg] u(t).
\end{align*}
This gives rise to 
\begin{equation}  \label{V' V}
V(t)+t\,V'(t)  = \mathsf{A}_2(t) u''(t) + \mathsf{A}_1(t) u'(t) + \mathsf{A}_0(t) u(t),
\end{equation}
where 
\begin{align*}
\mathsf{A}_2(t) &= -\frac{1+\tau}{2}\,t+\frac{3(1-\tau)}{t},
\\
\mathsf{A}_1(t) &=  -(1+\tau) \Big(  \frac{3\tau(1+\tau)}{4} \, t^2+ (1-\tau^2)(N-1) + \frac{2-6\tau+\tau^2}{2}     \Big) - \frac{3(1-\tau)}{t^2},
\\
\mathsf{A}_0(t) &= -\bigg[ \frac{\tau(1+\tau)}{4} \,t^2 + (1+\tau)  \,(N-1) + \frac{5+\tau}{4}     \bigg] \tau (1+\tau)^2\,t. 
\end{align*} 
Taking one more derivative and using Proposition~\ref{Prop_u ODE} once again, after simplifications as above, we obtain 
\begin{equation} \label{V' V''}
 2V'(t)+t\,V''(t)  = \mathsf{B}_2(t) u''(t) + \mathsf{B}_1(t) u'(t) + \mathsf{B}_0(t) u(t),
\end{equation} 
where 
\begin{align*}
\mathsf{B}_2(t) &=  -\frac{(1+\tau)^3(1-2\tau)}{4(1-\tau)}\,t^2 +(1+\tau)\Big( 1-3\tau+\tau^2-(N-1)(1-\tau^2) \Big)-\frac{12(1-\tau)}{t^2} , 
\\
\mathsf{B}_1(t) &= -  \frac{\tau (1+\tau)^3(3+2\tau) }{ 8  }\,t^3  - \frac{ (1+\tau)^2 }{ 2(1-\tau) } \Big( 1-4\tau+5\tau^2-5\tau^3  +(N-1) (1-\tau^2)(1+2\tau)  \Big)  \,t
\\
&\quad + (1+\tau)\Big(  2-15\tau+3\tau^2+3(N-1)(1-\tau^2) \Big) \frac{1}{t} + \frac{12(1-\tau)}{t^3} ,
\\
\mathsf{B}_0(t) &= \bigg[  -\frac{\tau (1+\tau)^2 }{ 8(1-\tau) }\,t^4- \frac{1+\tau}{8(1-\tau)} \Big( 5(1-\tau)+4(N-1)(1+\tau) \Big)t^2  + 2(N-1)(1+\tau)+\frac52-4\tau    \bigg] \tau (1+\tau)^2. 
\end{align*} 
Note that if $u$ is of the form \eqref{u abc V}, we have 
\begin{align}
\begin{split} \label{u V V' V'' v0}
t^2 u(t)& =\Big( \alpha(t) t^2-\beta(t) t+2\gamma(t) \Big) V(t)
+ \Big( \beta(t)\,t-2 \gamma(t) \Big)\Big( V(t)+t V'(t)\Big) + \gamma(t) t \Big( 2V'(t)+t V''(t) \Big). 
\end{split}
\end{align}
We now substitute \eqref{V in terms of u}, \eqref{V' V} and \eqref{V' V''} into \eqref{u V V' V'' v0}. 
Then by comparing the coefficients of the $u$, $u'$, and $u''$ terms, we derive a system of algebraic equations for $\alpha, \beta, $ and $\gamma$. Solving this system involves lengthy yet straightforward computations, ultimately yielding the explicit formulas of $\alpha$, $\beta$, and $\gamma$ in \eqref{u abc V}. This completes the proof.
\end{proof}


\appendix

\section{Integrable structure of the elliptic GinOE} \label{Appendix_integrable}

In this appendix, we give a brief exposition of the integrable structure of real eigenvalues of the elliptic GinOE established by Forrester and Nagao in \cite{FN08}, along with additional details on certain computations. 
In addition to this, we prove \eqref{M 2p w} and \eqref{M 2p w I0I1} in Remark~\ref{Rem_large N elliptic GinOE}.

Let us write 
\begin{equation} \label{Ck scaled Hermite}
C_k(x):= \Big( \frac{\tau}{2} \Big)^{k/2} H_k\Big( \frac{x}{ \sqrt{2\tau} } \Big)
\end{equation}
for the scaled monic Hermite polynomials.
Note that 
\begin{align} \label{Ck relations}
\frac{d}{dx} C_n(x)=n C_{n-1}(x), \qquad x\, C_n(x)= C_{n+1}(x)+n \tau \, C_{n-1}(x). 
\end{align}

It was shown in \cite[Theorem 1]{FN08} that
\begin{equation} \label{SOP Hermite}
p_{2k}(x):=C_{2k}(x), \qquad p_{2k+1}(x):=C_{2k+1}(x)-2k \, C_{2k-1}(x).
\end{equation}
form skew-orthogonal polynomials associated with real eigenvalues of the elliptic GinOE. 
This is indeed one of the very few examples of explicit construction of skew-orthogonal polynomials for asymmetric random matrices in the symmetry class of the GinOE. 
(See \cite{APS10} for another example in the context of the chiral GinOE.)
We also denote 
\begin{equation}
\begin{split}  \label{Phik Hermite}
\Phi_k(x) & :=\int_{\R} \, \textup{sgn}(x-y) \, p_k(y) e^{ -\frac{ y^2 }{ 2(1+\tau )} } \,dy
 = \int_\R p_k(y) e^{ -\frac{ y^2 }{ 2(1+\tau )} } \,dy -2 \int_x^{\infty}  p_k(y) e^{ -\frac{ y^2 }{ 2(1+\tau )} } \,dy.
\end{split}
\end{equation}
Then one can notice that
\begin{equation} \label{Phik derivative pk}
\frac{d}{dx} \Phi_k(x)= 2 \,p_k(x) e^{ -\frac{x^2}{ 2(1+\tau) } }. 
\end{equation}
From the general theory of the Pfaffian point process, it follows that all correlation functions of real eigenvalues of the elliptic GinOE can be expressed in terms of a Pfaffian of a skew-kernel constructed using \eqref{SOP Hermite} and \eqref{Phik Hermite}.
In particular, the density of real eigenvalues is given by 
\begin{equation} \label{RN Phi pk}
R_N(x)= \frac{ e^{-\frac{x^2}{2(1+\tau)} } }{  2\sqrt{2\pi} (1+\tau)  }  \sum_{k=0}^{N/2-1} \frac{1}{ (2k)! } \Big( \Phi_{2k}(x) p_{2k+1}(x)-\Phi_{2k+1}(x) p_{2k}(x) \Big), 
\end{equation}
see \cite[Eq. (6.6)]{FN08}.  
Each summation in this expression can be computed as follows, see \cite[Eqs. (6.7) and (6.11)]{FN08}. 
This leads to the expression of $R_N$ in \eqref{density eGinOE}.

\begin{prop}
For any even integer $N \ge 2$ and $\tau \in [0,1]$, we have
\begin{align}
\label{Hermite real cplx 1}
-\sum_{k=0}^{N/2-1} \frac{ \Phi_{2k+1}(x) p_{2k}(x)  }{(2k)!} & =2(1+\tau) e^{ -\frac{x^2}{ 2(1+\tau) }  }  \sum_{k=0}^{N/2-1} \frac{C_{2k}(x)^2}{(2k)!} ,
\\
 \label{Hermite real cplx 2}
\sum_{k=0}^{N/2-1} \frac{ \Phi_{2k}(x) p_{2k+1}(x) }{(2k)!} & =  2(1+\tau) e^{ -\frac{x^2}{2(1+\tau)} } \sum_{k=0}^{N/2-2} \frac{ C_{2k+1}(x)^2 }{(2k+1)!} + \frac{C_{N-1}(x) \Phi_{N-2}(x) }{ (N-2)! } .  
\end{align}
In particular, we have
\begin{equation} \label{Hermite real cplx}
R_N(x)= \frac{ e^{ -\frac{x^2}{1+\tau} } }{ \sqrt{2\pi} } \sum_{k=0}^{N-2} \frac{ C_k(x)^2 }{k!}  + \frac{ e^{ -\frac{x^2}{2(1+\tau)} }  }{ 2\sqrt{2\pi} (1+\tau) } \frac{ C_{N-1}(x)\Phi_{N-2}(x)  }{ (N-2)! } 
\end{equation}
and
\begin{equation} \label{RN integral}
\int_{\R} R_N(x)\,dx= \sqrt{ \frac{2}{\pi} } \sum_{k=0}^{N/2-1}  \frac{(\tau/2)^{2k} }{(2k)!}  \int_\R e^{ -\frac{ x^2 }{ 1+\tau }  }  H_{2k}\Big( \frac{x}{ \sqrt{2\tau} } \Big)^2 \,dx . 
\end{equation}
\end{prop}

\begin{proof}
Using \eqref{Ck relations}, one can observe that
\begin{align}
p_{2k+1}(x) & =-(1+\tau) e^{ \frac{x^2}{2(1+\tau)}  } \frac{d}{dx} \Big[ e^{ -\frac{x^2}{2(1+\tau)} } C_{2k}(x) \Big],    \label{Hermite relation 1}
\\
 p_{2k+2}(x)-(2k+1) p_{2k}(x) & =-(1+\tau) e^{ \frac{x^2}{ 2(1+\tau) }  } \frac{d}{dx} \Big[ e^{ -\frac{x^2}{ 2(1+\tau) }  } C_{2k+1}(x)  \Big].  \label{Hermite relation 2}
\end{align} 
Then by \eqref{Hermite relation 1}, we have 
\begin{align*}
&\quad \Phi_{2k+1}(x) p_{2k}(x)  = \int_\R \sgn(x-y) p_{2k+1}(y) p_{2k}(x) e^{ -\frac{ y^2}{ 2(1+\tau) } } \,dy
\\
&= -(1+\tau)  C_{2k}(x) \bigg( \int_{-\infty}^x  \frac{d}{dy} \Big[ e^{ -\frac{y^2}{2(1+\tau)} } C_{2k}(y)  \Big] \,dy -\int_{x}^\infty    \frac{d}{dy} \Big[ e^{ -\frac{y^2}{2(1+\tau)} } C_{2k}(y)  \Big] \,dy  \bigg)
= -2(1+\tau) e^{ -\frac{x^2}{ 2(1+\tau) }  } C_{2k}(x)^2,
\end{align*}
which gives \eqref{Hermite real cplx 1}. 
On the other hand, since 
\begin{align*}
\Phi_{2k}(x) p_{2k+1}(x) & = \int_\R \sgn(x-y) p_{2k}(y) p_{2k+1}(x) e^{ -\frac{ y^2}{ 2(1+\tau) } } \,dy 
\\
&=\Big( C_{2k+1}(x)-2k C_{2k-1}(x) \Big) \int_\R \sgn(x-y) C_{2k}(y) e^{ -\frac{ y^2}{ 2(1+\tau) } } \,dy, 
\end{align*}
we have 
\begin{align*}
&\quad \sum_{k=0}^{N/2-1} \frac{ \Phi_{2k}(x) p_{2k+1}(x) }{(2k)!}   = \sum_{k=0}^{N/2-1} \frac{ C_{2k+1}(x)-2k C_{2k-1}(x) }{(2k)!}  \int_\R \sgn(x-y) C_{2k}(y) e^{ -\frac{ y^2}{ 2(1+\tau) } } \,dy
\\
 & = \sum_{k=0}^{N/2-1} \frac{ C_{2k+1}(x) }{(2k)!}  \int_\R \sgn(x-y) C_{2k}(y) e^{ -\frac{ y^2}{ 2(1+\tau) } } \,dy  -  \sum_{k=0}^{N/2-2} \frac{ C_{2k+1}(x) }{(2k+1)!}  \int_\R \sgn(x-y) C_{2k+2}(y) e^{ -\frac{ y^2}{ 2(1+\tau) } } \,dy.
\end{align*}
It now follows from \eqref{Hermite relation 2} that 
\begin{align*}
&\quad \sum_{k=0}^{N/2-1} \frac{ \Phi_{2k}(x) p_{2k+1}(x) }{(2k)!}  - \frac{C_{N-1}(x) \Phi_{N-2}(x) }{ (N-2)! } 
\\
&=-  \sum_{k=0}^{N/2-2} \frac{ C_{2k+1}(x) }{(2k+1)!} \int_\R \sgn(x-y) \Big(C_{2k+2}(y)-(2k+1) C_{2k}(y)\Big) e^{ -\frac{ y^2}{ 2(1+\tau) } } \,dy
\\
&=(1+\tau)   \sum_{k=0}^{N/2-2} \frac{  C_{2k+1}(x)  }{(2k+1)!} \bigg( \int_{-\infty}^x \frac{d}{dy} \Big[ e^{ -\frac{y^2}{2(1+\tau)} } C_{2k+1}(y)  \Big] \,dy -\int_{x}^\infty    \frac{d}{dy} \Big[ e^{ -\frac{y^2}{2(1+\tau)} } C_{2k+1}(y)  \Big] \,dy \bigg). 
\end{align*}
This gives rise to \eqref{Hermite real cplx 2}. 
Now the expression \eqref{Hermite real cplx} immediately follows from \eqref{RN Phi pk}, \eqref{Hermite real cplx 1} and \eqref{Hermite real cplx 2}. 
Note that by \eqref{Phik derivative pk}, we have
\begin{equation}
2 e^{-\frac{x^2}{2(1+\tau)} }  \Big( \Phi_{2k}(x) p_{2k+1}(x)+\Phi_{2k+1}(x) p_{2k}(x) \Big)= \frac{d}{dx} \Big[ \Phi_{2k}(x) \Phi_{2k+1}(x) \Big].
\end{equation}
Therefore, we have
\begin{equation}
\sum_{k=0}^{N/2-1} \int_\R  e^{-\frac{x^2}{2(1+\tau)} } \frac{ \Phi_{2k}(x) p_{2k+1}(x) }{(2k)!} \,dx = - \sum_{k=0}^{N/2-1} \int_\R  e^{-\frac{x^2}{2(1+\tau)} } \frac{ \Phi_{2k+1}(x) p_{2k}(x)  }{(2k)!}  \,dx ,
\end{equation}
which gives \eqref{RN integral}. 
\end{proof}

\begin{rem}
We also note that the expression \eqref{M0 expected number} follows from \eqref{RN integral} and 
\begin{align}
\begin{split}
&\quad \int_\R e^{-\frac{x^2}{1+\tau}} H_{2k}\Big( \frac{x}{\sqrt{2\tau}} \Big)^2 \,dx=\sqrt{2\tau} \int_\R e^{-\frac{2\tau }{1+\tau}u^2} H_{2k}( u )^2 \,du
\\
&=\sqrt{2\tau} \, 2^{ 2k-\frac{1}{2} } \Big( \frac{\tau}{1+\tau} \Big)^{ -2k-\frac{1}{2} } \Big(\frac{1-\tau}{1+\tau} \Big)^{ 2k } \Gamma\Big(2k+\frac{1}{2}\Big) {}_2F_1\Big(-2k,-2k;\frac{1}{2}-2k;-\frac{\tau}{1-\tau}\Big)
\\
&=\Big( \frac{1+\tau}{1-\tau} \Big)^{\frac12} \Big( \frac{\tau}{2} \Big)^{ -2k } \Gamma\Big( 2k+\frac{1}{2}\Big) {}_2F_1\Big(\frac{1}{2},\frac{1}{2};\frac{1}{2}-2k;-\frac{\tau}{1-\tau}\Big),
	\end{split}
\end{align}
where we have used \cite[(7.374-5)]{GR14} and the Euler's transformation \cite[Eq. (15.8.1)]{NIST}.
\end{rem}

Now we prove \eqref{M 2p w} and \eqref{M 2p w I0I1}. 
Before the proof, we mention that the function $M_0^{\rm w}$ also appears in a seemingly different context of the number variance of the GinUE, see \cite[Proposition 2.4]{ABES23}.
For the derivations, we make use of the power series expansion of the error function 
\begin{equation} \label{erf power}
	\erf(z)=\frac{2}{\sqrt{\pi}} \sum_{n=0}^{\infty}\frac{(-1)^n }{ n! \, (2n+1) } z^{2n+1},
\end{equation}
see \cite[Eq. (7.6.1)]{NIST}. 
By combining \eqref{erf power} with the Euler beta integral
\begin{equation}
\int_0^1 s^{2p}\,(1-s^2)^{n+\frac12} \,ds= \frac{ \Gamma(n+\frac32) \Gamma(p+\frac12) }{2\, \Gamma(n+p+2) }, 
\end{equation} 
we obtain 
\begin{align*}
\frac{1}{2\alpha\sqrt{\pi}} \int_{-2}^2 x^{2p}\,\erf\Big( \frac{\alpha}{2} \sqrt{4-x^2}\Big) \,dx &=  \frac{1}{\alpha \pi} \sum_{n=0}^{\infty}\frac{(-1)^n }{ n! \, (2n+1) }  \Big( \frac{\alpha}{2} \Big)^{2n+1}  \int_{-2}^2 x^{2p}\, (4-x^2)^{n+\frac12}\,dx
\\
&=  \frac{1}{ \pi} \sum_{n=0}^{\infty}\frac{(-1)^n  \alpha^{2n}  2^{2p+1}  }{ n! \, (2n+1) }      \frac{ \Gamma(n+\frac32) \Gamma(p+\frac12) }{ \Gamma(n+p+2) }.
\end{align*}
This gives rise to \eqref{M 2p w}.
On the other hand, \eqref{M 2p w I0I1} can be derived using the expansions
\begin{align} 
I_0\Big( \frac{\alpha^2}{2} \Big) e^{-\frac{\alpha^2}{2}}  = \sum_{k=0}^\infty \frac{(2k-1)!!}{(k!)^2} (-1)^k \Big(\frac{\alpha^2}{2}\Big)^k, \qquad 
I_1\Big( \frac{\alpha^2}{2} \Big)e^{-\frac{\alpha^2}{2}}  =\sum_{k=0}^\infty \frac{(2k-1)!!}{(k-1)!(k+1)!} (-1)^{k+1}\Big(\frac{\alpha^2}{2}\Big)^k, 
\end{align}
both of which follow from the definition \eqref{I nu} of the modified Bessel function.



\end{document}